\documentclass[12pt]{amsart}
\usepackage{amsmath,amsthm,latexsym,amsfonts,amssymb}
\usepackage{datetime}
\usepackage{amsrefs}
\usepackage[nameinlink,capitalize]{cleveref}

\newtheoremstyle{slthm}
{9pt}
{5pt}
{\slshape}
{}
{\bfseries}
{.}
{.5em}
{\thmname{#1}\thmnumber{ #2}\thmnote{ (#3)}}

\newtheoremstyle{prcl}
{9pt}
{5pt}
{\slshape}
{}
{\bfseries}
{.}
{.5em}
{\thmname{#3}\thmnumber{ #2}}

\newtheoremstyle{prblm}
{9pt}
{5pt}
{\rm}
{}
{\bfseries}
{.}
{.5em}
{\thmname{#3}\thmnumber{ #2}}

\theoremstyle{slthm}
\newtheorem{thm}{Theorem}[section]
\newtheorem{lemma}[thm]{Lemma}
\newtheorem{prop}[thm]{Proposition}
\newtheorem{cor}[thm]{Corollary}

\theoremstyle{definition}

\newtheorem{df}[thm]{Definition}
\newtheorem{dfs}[thm]{Definitions}

\newtheorem{nrmk}[thm]{Remark}
\newtheorem{nrmks}[thm]{Remarks}
\newtheorem{expl}[thm]{Example}
\newtheorem{expls}[thm]{Examples}

\theoremstyle{remark}
\newtheorem*{rmk}{Remark}
\newtheorem*{rmks}{Remarks}

\theoremstyle{prcl}
\newtheorem*{prclaim}{Proclaim}
\newtheorem{nprclaim}[thm]{Proclaim}

\theoremstyle{prblm}

\newcounter{flexnummark}

\DeclareMathOperator{\mm}{\mathfrak{m}}
\DeclareMathOperator{\nn}{\mathfrak{n}}

\DeclareMathOperator{\cl}{cl}

\DeclareMathOperator{\im}{Im}
\DeclareMathOperator{\re}{Re}
\DeclareMathOperator{\supp}{supp}
\DeclareMathOperator{\ord}{ord}

\newcommand{\mdots}{\dots}

\newcommand{\rest}[1]{\!\!\upharpoonright_{#1}}

\newcommand{\into}{\longrightarrow}

\renewcommand{\hat}{\widehat}
\renewcommand{\tilde}{\widetilde}
\renewcommand{\bar}{\overline}

\def\Ind#1#2{#1\setbox0=\hbox{$#1x$}\kern\wd0\hbox to 0pt{\hss$#1\mid$\hss}
	\lower.9\ht0\hbox to 0pt{\hss$#1\smile$\hss}\kern\wd0}

\def\Notind#1#2{#1\setbox0=\hbox{$#1x$}\kern\wd0\hbox to 0pt{\mathchardef
		\nn=12854\hss$#1\nn$\kern1.4\wd0\hss}\hbox to
	0pt{\hss$#1\mid$\hss}\lower.9\ht0 \hbox to
	0pt{\hss$#1\smile$\hss}\kern\wd0}

\newcommand{\norm}[1]{\left\|#1\right\|}
\newcommand{\set}[1]{\left\{#1\right\}}

\newcommand{\NN}{\mathbb{N}}
\newcommand{\ZZ}{\mathbb{Z}}

\newcommand{\QQ}{\mathbb{Q}}
\newcommand{\RR}{\mathbb{R}}

\newcommand{\CC}{\mathbb{C}}

\newcommand{\curly}[1]{\mathcal{#1}}
\newcommand{\A}{\curly{A}}

\newcommand{\C}{\curly{C}}

\newcommand{\G}{\curly{G}}

\newcommand{\T}{\curly{T}}

\renewcommand{\mm}{\mathbf{m}}
\renewcommand{\nn}{\mathbf{n}}

\newcommand{\bo}{\curly{B}}
\newcommand{\la}{\curly{L}}

\DeclareMathOperator{\Ranexp}{\RR_{an,exp}}

\newcommand{\Ranstar}{\RR_{\text{an}^*}}

\newcommand{\Ps}[2]{\mathbb{#1}\left[\!\left[#2\right]\!\right]}
\newcommand{\Pc}[2]{\mathbb{#1}\left\{#2\right\}}

\newcommand{\realgamma}{\Gamma\rest{(0,+\infty)}}
\newcommand{\realzeta}{\zeta\rest{(1,+\infty)}}

\allowdisplaybreaks[2]
\numberwithin{equation}{section}

\title {Multisummability for generalized power series}

\author {Jean-Philippe Rolin, Tamara Servi and Patrick Speissegger}

\address{ Institut de Math\'ematiques de Bourgogne \\
	Universit\'e de Bourgogne Franche-Comt{\'e}\\
	UMR 5584, CNRS\\
	B.P. 47870 \\
	21078 Dijon, France}
\email{jean-philippe.rolin@u-bourgogne.fr}

\address {Institut de Math\'ematiques de Jussieu -- Paris Rive Gauche \\
	Universit\'{e} Paris Cit\'{e} and Sorbonne Universit\'{e}, CNRS, IMJ-PRG, F-75013 Paris, France}

\email {servi@math.univ-paris-diderot.fr}

\address {Department of Mathematics and Statistics, McMaster University, 1280
Main Street West, Hamilton, Ontario L8S 4K1, Canada}

\email {speisseg@math.mcmaster.ca}

\date{\today\ at \currenttime}

\subjclass {Primary 40C10, 03C64, 26E10; Secondary 30D60}

\keywords {Gamma function, Zeta function, multisummability, quasianalyticity, \hbox{o-minimality}}

\thanks{The third author is supported by NSERC of Canada grant RGPIN 2018-06555. The authors would like to thank the Fields Institute for Research in Mathematical Sciences for its hospitality and financial support, as part of this work was done while at its Thematic Program on Tame Geometry and Applications. They would also like to thank the referee for many useful remarks.}

\begin{document}

\begin{abstract}
	We develop multisummability, in the positive real direction, for generalized power series with natural support, and we prove o-minimality of the expansion of the real field by all multisums of these series.  This resulting structure expands both $\RR_{\G}$ and the reduct of $\Ranstar$ generated by all convergent generalized power series with natural support; in particular, its expansion by the exponential function defines both the Gamma function on $(0,\infty)$ and the Zeta function on $(1,\infty)$.
\end{abstract}

\maketitle
\markboth{J.-P. Rolin, Tamara Servi and Patrick Speissegger}{Amalgamating Gamma and Zeta}

\section*{Introduction}

 We generalize the theory of multisummability in the positive real direction, as discussed in \cite{Balser:2000fk,Tougeron:1994fk,Dries:2000mx}, to certain non-convergent power series with real non-negative exponents (introduced in  \cite[p. 4377]{Dries:1998xr}). Examples of such series are Dirichlet series (after the change of variables $s=-\log x$), and asymptotic expansions of certain solutions of differential equations \cite{wasow:asymptotic} and of certain functions appearing in Dulac's problem \cite{ilyashenko:dulac}. 
 
 Our main motivation here comes from o-minimality: summation processes induce a quasianalyticity property which is usually needed to prove that a given structure is o-minimal. In their paper \cite{Dries:1997jl}, Van den Dries, Macintyre and Marker show that neither Euler's Gamma function $\Gamma$ restricted to $(0,+\infty)$, nor the Riemann Zeta function $\zeta$ restricted to $(1,+\infty)$, are definable in the o-minimal structure $\Ranexp$ \cite[Theorem 5.11 and Corollary 5.14]{Dries:1997jl}.  Subsequently, Van den Dries and Speissegger constructed the o-minimal expansions $(\Ranstar,\exp)$ \cite{Dries:1998xr,Dries:2000mx} and $(\RR_\G,\exp)$ \cite{Dries:2000mx}, and they proved that $\zeta\rest{(0,+\infty)}$ is definable in the former, but not in the latter \cite[Corollary 10.11]{Dries:2000mx}, while $\Gamma\rest{(0,+\infty)}$ is definable in the latter \cite[Example 8.1]{Dries:2000mx}.  At the time, it was unknown whether $\Gamma\rest{(0,+\infty)}$ was definable in the former.

This state of affairs thus left the following question unanswered: \textsl{is there an o-minimal expansion of the real field in which both $\realgamma$ and $\realzeta$ are definable?}  Based on additional information gained from Rolin and Servi's paper \cite{MR3349791} about the structures $(\Ranstar,\exp)$ and $(\RR_\G,\exp)$, we show in a separate paper (in preparation) that $\realgamma$ is not definable in $(\Ranstar,\exp)$ either.  So to answer the question in the affirmative, we need to come up with an o-minimal structure that properly expands both the expansion of the real field by $\realgamma$ and  the expansion of the real field by $\realzeta$.

Indeed, we construct here an o-minimal expansion of the real field that expands $(\RR_\G,\exp)$ and in which $\realzeta$ is definable (see the Main Corollary below). 

To recap, for an indeterminate $X = (X_1, \dots, X_n)$, we denote by $\Ps{C}{X^*}$ the set of all \textbf{generalized power series} of the form $F(X) = \sum_{\alpha \in [0,\infty)^n} a_\alpha X^\alpha$, where each $a_\alpha \in \CC$ and the \textbf{support} $$\supp(F):= \set{\alpha \in [0,\infty)^n:\  a_\alpha \ne 0}$$ is contained in a product $A_1 \times \cdots \times A_n$ of sets $A_i \subset [0,\infty)$ that are well-ordered with respect to the usual ordering of the real numbers (see \cite[Section 4]{Dries:1998xr} for details).  The series $F(X)$ \textbf{converges} if there exists $r>0$ such that $\|F\|_r:= \sum_\alpha |a_\alpha|r^\alpha < \infty$; we denote by $\Pc{C}{X^*}$ the set of all convergent generalized power series \cite[Section 5]{Dries:1998xr}.

The generalized power series that we extend the notion of multisummability to have special support: we call a set $A \subseteq \RR$ \textbf{natural} if $A \cap (-\infty,a)$ is finite, for every $a \in \RR$; and we call a set $A \subseteq \RR^n$ \textbf{natural} if $A \subseteq A_1 \times \cdots \times A_n$ with each $A_i \subseteq \RR$ natural. Restricting our
attention to generalized power series with natural support allows us to use such objects as
asymptotic expansions of germs (see Proposition \ref{gevrey_estimates}). This has already been expoited in \cite{KRS}, where the o-minimality of the expansion of the real field by certain Dulac germs is proven.

In \cref{one-variable_chapter,several-variable_chapter} below, we define a notion of multisummability in the positive real direction for generalized power series of natural support, appropriately named \textit{generalized multisummability in the positive real direction} (or simply \textit{generalized multisummability in the real direction} when working in the logarithmic chart of the Riemann surface of the logarithm, as we do throughout this paper).  We verify that the resulting system $\G^*$ of algebras (both of functions and of germs) satisfies the axioms in \cite{MR3349791}, leading to the following:  let the language $\la_{\G^*}$ and the structure $\RR_{\G^*}$ be as in \cite[Definition 1.21]{MR3349791} for our system $\G^*$ of algebras in place of $\A$ there.

\begin{prclaim}[Main Theorem]
	\begin{enumerate}
		\item The structure $\RR_{\G^*}$ is model complete, o-minimal and polynomially bounded and has field of exponents $\RR$.
		\item The structure $\RR_{\G^*}$ admits quantifier elimination in the language $\la_{\G^*} \cup \{(\cdot)^{-1}\}$.
	\end{enumerate}
\end{prclaim}

By construction, all functions defined on compact polydisks by convergent generalized power series with natural support are definable in $\RR_{\G^*}$; and we show in  \cref{one-variable_chapter} that the same holds for all functions defined on compact sets by standard power series that are multisummable in the positive real direction. Recall that, for $x\in\left[0,e^{-2}\right]$, $\zeta\left(-\log x\right)$ is the sum of the generalized power series $\sum_{n=1}^{\infty}x^{\log n}$,
which has natural support. In particular, both $\exp\rest{[0,1]}$ and $\zeta(-\log x)\rest{[0,e^{-2}]}$ are definable in $\RR_{\G^*}$, as is the function $\log\Gamma(x) - (x-\frac12)\log x$ on the interval $(1,+\infty)$ (see \cite[Example 8.1]{Dries:2000mx}).  Therefore, Theorem B of \cite{Dries:2000mx} gives:

\begin{prclaim}[Main Corollary]
	\begin{enumerate}
		\item The structure $(\RR_{\G^*},\exp)$ is model complete and o-minimal, and it admits quantifier elimination in the language $\la_{\G^*} \cup \{\exp,\log\}$. 
		\item The functions $\realgamma$ and $\realzeta$ are definable in $(\RR_{\G^*},\exp)$. \qed
	\end{enumerate}
\end{prclaim}

As we rely on \cite{MR3349791} for the proof of o-minimality of $\RR_{\G^*}$, the main contribution of this paper is the generalization of multisummability in the positive real direction to generalized power series of natural support and the establishment of the axioms in \cite{MR3349791} for the corresponding system $\G^*$ of algebras of functions and germs.   

As in \cite{Dries:2000mx}, our starting point here is a characterization, due to Tougeron \cite{Tougeron:1994fk}, of multisummable power series in terms of infinite sums of convergent power series of decreasing radii of convergence.  Thus, we move to the logarithmic chart of the Riemann surface of the logarithm, since we are working with arbitrary real exponents.  Then we define a multisummable \textit{generalized} power series (in the real direction) as the infinite sum of a sequence of convergent \textit{generalized} power series with decreasing radii of convergence and support contained in a fixed natural set (\cref{pre_gen_multi_section}).  

The corresponding theory of multisummability in one variable, developed in \cref{one-variable_chapter}, differs from the classical one in that there is no origin around which we can use contour integration.  One example of a classical result that we cannot generalize is the following: every classical multisummable power series can be decomposed into a sum of singly summable series; we do not know if this is the case in the generalized setting (see \cref{question_1} for details).  However, we do obtain the crucial quasianalyticity for our system of algebras (\cref{quasi}).

Also, as in \cite{Dries:2000mx}, this approach lends itself naturally to define generalized multisummability in the positive real direction in several variables, and we follow the corresponding steps in \cite{Dries:2000mx} as closely as possible (\cref{several-variable_chapter}).  In \cref{substitution,o-min_chapter}, we establish the axioms of \cite{MR3349791}.

\begin{rmks}
	(1) To our knowledge, this is the first time \cite{MR3349791} was used to prove the o-minimality of a structure that was previously unknown to be o-minimal.  The same procedure could be used to obtain the o-minimality (and related results) of the structures $\Ranstar$ \cite{Dries:1998xr}, $\RR_\G$ \cite{Dries:2000mx} and $\RR_\C$ \cite{MR1992825}.  The resulting quantifier elimination given by \cite[Theorem B]{MR3349791} is new in each of these cases, and it is used in our forthcoming paper to show that $\realgamma$ is not definable in $(\Ranstar,\exp)$.  
	
	(2) The only closure property needed in \cite{MR3349791} but not established in \cite{Dries:1998xr},  \cite{Dries:2000mx} or \cite{MR1992825} is closure under infinitesimal substitutions in the convergent variables (\cref{infinitesimal_prop}).  The proof of this in the structures discussed in the previous remark is similar to the proof given here for $\G^*$.
\end{rmks}

Finally, from the point of view of generalized multisummability, as in the classical theory, there is nothing special about the positive real direction.  Our generalized notion works in any other direction, and one could correspondingly come up with a notion of ``generalized multisummability'' as done in the classical situation.  This raises some interesting questions in their own right (see \cref{question_1}), which we do not address in this paper.

\section{Generalized multisummable functions of one variable} \label{one-variable_chapter}

\subsection{Preliminaries} \label{prelim1}

We denote by $$\bar\CC = \CC \cup \{-\infty\}$$ the logarithmic chart of the Riemann surface of the logarithm, with the additional ``origin'' of $\bar\CC$ represented by ``$-\infty$'', where we convene that $\re(-\infty) = -\infty$. For $\frak r \in \RR$, we let 
\begin{equation*}
	H(\frak r):= \set{u+iv \in \bar\CC:\ u < \frak r}
\end{equation*}
be the \textbf{log-disk} of \textbf{log-radius} $\frak r$.  For $d,\frak r \in \RR$, a \textbf{log-sector} is a set 
\begin{equation*}
	S(d,\frak r,\theta):= \begin{cases} \set{u+iv \in \CC:\ u < \frak r, |d-v| < \theta} \cup \{-\infty\} &\text{if } \theta \in (0,\infty), \\ H(\frak r) &\text{if } \theta = \infty, \end{cases}
\end{equation*}
and a \textbf{log-line} is a set $$T(d):= \set{u+iv \in \CC:\ v = d} \cup \{-\infty\}.$$  (We shall mainly focus on the direction $d=0$ in this paper.)  We extend the standard topology on $\CC$ to $\bar\CC$ by declaring the log-disks as basic open neighbourhoods of $-\infty$.  Note that the usual covering map of the Riemann surface of the logarithm is represented in the logarithmic chart by the exponential function, and we extend it to a continuous function on $\bar\CC$ by setting $e^{-\infty}:= 0$.  For each $d \in \RR$, the restriction of $e^w$ to $S(d,\infty,\pi) \setminus\{-\infty\}$ is injective; its inverse is the \textbf{branch of the logarithm $\log_d$ in the direction $d$}.

We are mostly interested in partial functions on $\bar\CC$ with values in $\CC$.  In this spirit, we call a set $D \subseteq \bar\CC$ a \textbf{log-domain} if $D \cap \CC$ is a domain (in particular, every domain in $\CC$ is a log-domain).  If $D \subseteq \bar\CC$ is a log-domain, a \textbf{log-holomorphic function on $D$} is a continuous function $f:D \into \CC$ such that the restriction of $f$ to $D \cap \CC$ is holomorphic.  For example, every holomorphic function on a domain in $\CC$ is log-holomorphic, and the exponential function is log-holomorphic on $\bar\CC$.

\subsection{The logarithmic Borel and Laplace transforms}\label{log_Borel_section}

\subsubsection{Logarithmic Borel transform} \label{lbt}
Let $d,\frak r \in \RR$ and $\theta > \pi/2$, and write $S = S(d,\frak r,\theta)$.  Let $f:S \into \CC$ be such that $f\rest{\bar S_0}$ is bounded and log-holomorphic, for every closed log-subsector $\bar S_0$ of $S$.
Given a closed log-subsector $\bar S_0 = \cl(S(d', \frak r',\theta'))$ of $S$ with $\theta'>\frac\pi2$, denote by $\partial \bar S_0$ the directed path following the boundary of $\bar S_0$ from the ``lower left end'' to the ``upper left end''.  We define the \textbf{logarithmic Borel transform} $\bo_{d'} f: T(d') \into \CC$ in the direction $d'$ of $f$ by
\begin{equation*}
	\bo_{d'} f(w) := \frac{e^w}{2\pi i} \int_{\partial \bar S_0} e^{e^{w-\eta}} f(\eta) \frac{d\eta}{e^\eta}.
\end{equation*}
We leave it as an exercise to check that $\bo_{d'} f$ only depends on $d'$, but not on the other parameters of $\bar S_0$ (as long as they are in the prescribed range).  More is true:

\begin{nrmk}\label{Loray_rmk}
	If $\theta' < \pi$ and $g(z):= f(\log_d z)$, for $z \in \exp(\bar S_0)$, then the change of variables $z = e^w$  gives that 
	\begin{equation*}
		\left(\bo_{d'} f\right)(\log_d z) = z \cdot \left(\bo_{d'}g\right) (z),
	\end{equation*}
	where $\bo_{d'} g$ denotes the Borel transform of $g$ in the direction $d'$ as defined in \cite{loraynotes} (see also Section 5.2 of \cite{Balser:2000fk}).  Thus, the following is obtained from Propri\'et\'es 1--3 on p. 38 of \cite{loraynotes}:
\end{nrmk}

\begin{prop}\label{Borel_prop}
	Set $S':= S\left(d,\infty,\theta-\frac\pi2\right)$.
	\begin{enumerate}
		\item The function $\bo f: S' \into \CC$ defined by $\bo f(w):= (\bo_{\im w} f)(w)$ is log-holomorphic on every closed log-subsector $\bar S_0$ of $S'$.
		\item For every closed log-subsector $\bar S_0$ of $S'$, there exist $C,D>0$ such that $$|\bo f(w)| \le C e^{De^{\re w}} \quad\text{for } w \in \bar S_0.$$
		\item Let $\alpha \ge 0$, and assume that for every closed log-subsector $\bar S_0$ of $S$, we have $|f(w)| = O\left(e^{\alpha\re w}\right)$ as $w \to -\infty$ in $\bar S_0$.  Then for every closed log-subsector $\bar S_0$ of $S'$, we have $|\bo f(w)| = O\left(e^{\alpha\re w}\right)$ as $w \to -\infty$ in $\bar S_0$.\qed
	\end{enumerate}
\end{prop}

Accordingly, we call the function $\bo f$ defined in the proposition above the \textbf{log-Borel transform} of $f$.

For $D \subseteq \bar\CC$ and $g:D \into \CC$, we set $$\|g\|_D := \sup\set{|g(z)|:\ z \in D}.$$
For later use, we make the bound in \cref{Borel_prop}(2) more precise:

\begin{lemma}\label{lbt_sup_estimate}
	Let $\bar S_0 = \cl(S(d', \frak r',\theta'))$ be a closed subsector of $S$ with $\theta' \in \left(\frac\pi2,\theta\right)$, and set $S':= S\left(d',r,\theta' - \frac\pi2\right)$ and $C:= \sin\left(\frac{\theta-\theta'}2\right)$.
	Then $$\left\|\bo f\right\|_{S'} \le \begin{cases} \frac{\left\| f\right\|_{\bar S_0}}{C} e &\text{if } r \le \frak r', \\ \frac{\left\| f\right\|_{\bar S_0}}{C} e^{e^{r-\frak r'}} e^{r - \frak r'} &\text{if } r \ge \frak r'. \end{cases}$$
\end{lemma}

\begin{proof}
	Let $w \in S'$; we compute $\bo f(w)$ by computing $\bo_d f(w)$, where $d:= \im w$ and the integral is taken along the contour $\delta:= \partial S\left(d,\rho,\alpha\right)$, where $\alpha:= \frac{\theta-\theta'+\pi}2$ and $\rho:= \min\{\re w, \frak r'\}$.  For $\eta \in \delta$, we distinguish two cases.\smallskip
	
	\noindent\textbf{Case 1:} $|\im(w-\eta)| = \alpha$.  Then $\re(e^{w-\eta}) = \cos\alpha \cdot e^{\re w-\re\eta}$; since $C = -\cos\alpha$, we get 
	\begin{align*}
	\frac1{2\pi} \left|e^w\int_{|\im(w-\eta)| = \pi} e^{e^{w-\eta}} f(\eta) \frac{d\eta}{e^\eta}\right| &\le \frac{\left\|f\right\|_{\bar S_0}}{2\pi} \int_{|\im(w-\eta)| = \pi} e^{\cos\alpha \cdot e^{\re w-\re\eta}} e^{\re w -\re\eta} d\eta \\
	&= \frac{\left\|f\right\|_{\bar S_0}}{\pi} \int_{-\infty}^{\rho} e^{\cos\alpha \cdot e^{\re w -r}} e^{\re w-r} dr \\
	&= \frac{\left\|f\right\|_{\bar S_0}}{\pi C} e^{\cos\alpha\cdot e^{\re w-r}}\Big|_{-\infty}^{\rho} \\
	&\le \frac{\left\|f\right\|_{\bar S_0}}{C},
	\end{align*}
	because $\cos\alpha < 0$.
	\smallskip
	
	\noindent\textbf{Case 2:} $\re\eta = \rho$.  Then we have $$\re(e^{w-\eta}) \le |e^{w-\eta}| = e^{\re w-\re\eta} = e^{\re w-\rho},$$ so that
	\begin{equation*}
	\frac1{2\pi} \left|e^w\int_{\re\eta = \rho} e^{e^{w-\eta}} f(\eta) \frac{d\eta}{e^\eta}\right| 
	\le \left\|f\right\|_{\bar S_0}  e^{e^{\re w-\rho}} e^{\re w - \rho}  \le \begin{cases} \|f\|_{\bar S_0} e &\text{if } \re w \le \frak r', \\ \|f\|_{\bar S_0} e^{e^{\re w-\frak r'}} e^{\re w - \frak r'} &\text{if } \re w \ge \frak r'. \end{cases}
	\end{equation*}
	Combining the two cases we obtain the lemma.
\end{proof}

\subsubsection{Logarithmic Laplace transform} \label{llt}

We fix an arbitrary direction $d \in \RR$.  Let $f:T(d) \into \CC$ be continuous, and assume that there exist $C,D>0$ such that $$|f(w)| \le Ce^{De^{\re w}} \quad\text{for all } w \in T(d).$$  We let $$U(d,D) := \set{z \in \CC\setminus\{0\}:\ \cos(\arg z - d) > D|z|} \cup \{0\}$$ be the \textbf{Borel disk} of diameter $\frac1D$ touching the origin and centered on the ray in direction $d$.  Correspondingly, we let $$V(d,D):= \set{w \in \CC:\ \cos(\im w - d) > De^{\re w}} \cup \{-\infty\}$$ the \textbf{log-Borel disk in the direction $d$ of extent} $-\log D$; note indeed that $U(d,D) = \exp(V(d,D))$.  We define the \textbf{log-Laplace transform} $\la_df:V(d,D) \into \CC$ in the direction $d$ of $f$ by
$$\la_d f(w):= \int_{T(d)} e^{-e^{\eta-w}} f(\eta) d\eta.$$

\begin{nrmk}\label{Loray_rmk2}
	If $g(z):= f(\log_d z)$, for $z \in \CC$ such that $\arg z = d$, then the change of variables $z = e^w$  gives that 
	\begin{equation*}
	\left(\la_{d} f\right)(\log_d z) = \frac{\left(\la_{d}g\right) (z)}z,
	\end{equation*}
	where $\la_{d} g$ denotes the Laplace transform of $g$ in the direction $d$ as defined in \cite{loraynotes} (see also Section 5.1 of \cite{Balser:2000fk}).  Thus, the following is obtained from Propri\'et\'es 1--2 on pp. 41--42 of \cite{loraynotes}:
\end{nrmk}

\begin{prop}\label{Laplace_prop}
	Let $\varphi > 0$ and set $S:= S(d,\infty,\varphi)$.  Let $f:S \into \CC$, and assume that for every closed log-subsector $\bar S_0$ of $S$, the restriction $f\rest{\bar S_0}$ is log-holomorphic and there exist $C,D>0$ such that $|f(w)| \le Ce^{De^{\re w}}$ for $w \in \bar S_0$.  Then:
	\begin{enumerate}
		\item For each $\theta \in (0,\varphi)$, there exists $0 < R(\theta) \le \frac1D$ such that $\la_d f$ has a log-holomorphic extension $\la f: V(d,R(\theta)) \into \CC$.
		\item Let $\alpha \ge 0$, and assume that for every closed log-subsector $\bar S_0$ of $S$, we have $|f(w)| = O\left(e^{\alpha\re w}\right)$ as $w \to -\infty$ in $\bar S_0$.  Then, in the situation of part (1), for every closed log-subsector $\bar S_0$ of $V(d,R(\theta))$, we have $|\la f(w)| = O\left(e^{\alpha\re w}\right)$ as $w \to -\infty$ in $\bar S_0$.\qed
	\end{enumerate}
\end{prop}

In view of the previous proposition, we call the union $V:=\bigcup_{\theta\in(0,\varphi)}V(d,R(\theta))$ a \textbf{log-sectorial domain}, and we refer to the common extension $\la f:V \into \CC$ of $\la_0 f$ of $\la_0 f$ given by \cref{Laplace_prop} as the \textbf{log-Laplace transform} of $f$.  Note that, in practice, we shall usually restrict the domain of $\la f$ to a sector $S\left(d,\log R,\theta+\frac\pi2\right)$ for suitable $\theta \in (0,\varphi)$ and $R>0$ on which it is log-holomorphic.

For $f:S(d,\frak r,\theta) \into \CC$ as in \cref{lbt}, \cref{Borel_prop} implies that $\la(\bo f)$ is defined and log-holomorphic on every closed log-subsector $\bar S_0$ of $S(d,\frak r,\theta)\cap V$.  Indeed, $\la$ is the inverse operator to $\bo$ (see page 44 of \cite{loraynotes}):

\begin{prop} \label{llt-lbt_prop}
	For $f:S(d,\frak r,\theta) \into \CC$ as in \cref{lbt}, we have $\la(\bo f) = f$ on $S(d,\frak r,\theta)\cap V$. \qed
\end{prop}

\begin{expl}\label{power_expl}
	For $\alpha \in \RR$, we set $p_\alpha(w):= e^{\alpha w}$.  Then for $w \in \RR$, we have
	\begin{align*}
	\la_0(p_\alpha)(w) &= \int_{-\infty}^{\infty} e^{-e^{\eta-w}} e^{\alpha\eta} d\eta & \\ &=\int_0^\infty e^{-\zeta/e^w} \zeta^{\alpha-1} d\zeta &\text{(taking } \zeta = e^\eta) \\ &= e^{\alpha w} \int_0^\infty e^{-\xi} \xi^{\alpha-1} d\xi &\text{(taking } \zeta = e^w\eta) \\ &= \Gamma(\alpha) e^{\alpha w}.
	\end{align*}
	It follows, by analytic continuation and \cref{Laplace_prop}, that $\la(p_\alpha) = \Gamma(\alpha) p_\alpha$, and hence by \cref{llt-lbt_prop} that $\bo(p_\alpha) = \frac{p_\alpha}{\Gamma(\alpha)}$.
\end{expl}

\subsection{Generalized power series with complex coefficients}\label{lbt_convergent}

Let now $F(X) = \sum_{\alpha \ge 0} a_\alpha X^\alpha$ $\in \Pc{C}{X^*}$ be such that $\|F\|_r < \infty$, for some $r>0$. We explain here how such a series defines a log-holomorphic function on some log-disk.
Denoting by $\log$ the principle branch of the logarithm on $\CC \setminus (-\infty,0]$, we set $$z^\alpha := e^{\alpha\log z} \quad\text{for } z \in \CC \setminus (-\infty,0].$$  Then, for $w \in S(0,\infty,\pi)$, we have that $$p_\alpha(w) = (e^w)^\alpha;$$  in other words, the entire function $p_\alpha$ extends the function $w \mapsto (e^w)^\alpha: S(0,\infty,\pi) \into \CC$.  Since $|p_\alpha(w)| = e^{\alpha\re w}$, it follows that the series $$\bar F(w) := \sum a_\alpha e^{\alpha w}$$ converges absolutely and uniformly, for $w \in H(\log r)$.  By Weierstrass's Theorem, the function $\bar F:H(\log r)\setminus\{-\infty\} \into \CC$ is holomorphic, and by the previous remarks, we have $$\bar F(w) = F(e^w) \quad\text{for } w \in S(0,\log r,\pi).$$  It follows, in particular, from \cite[Lemma 5.5]{Dries:1998xr} that $\bar F$ extends continuously to $-\infty$ and satisfies $\bar F(-\infty) = F(0)$.  Below, we refer to the log-holomorphic function $\bar F$ thus defined on $H(\log r)$ as the \textbf{log-sum} of $F(X)$.

\subsubsection{Logarithmic Borel transform of convergent generalized power series with natural support}

We again fix $F(X) = \sum a_\alpha X^\alpha \in \Pc{C}{X^*}$ and $r > 0$ such that $\|F\|_r < \infty.$  In addition, we assume that the support of $F(X)$ --- a subset of $[0,\infty)$ by definition --- is natural.  Since $\bar F$ is defined on $H(\log r) = S(0,\log r,\infty)$, we obtain from  \cref{Borel_prop} that its log-Borel transform $\bo\bar F$ is log-holomorphic on $\bar\CC$.

In view of \cref{power_expl}, we set $$\bo F(X):= \sum \frac{a_\alpha}{\Gamma(\alpha)} X^\alpha,$$ called the \textbf{formal Borel transform} of $F(X)$.  Note that, for $\sigma >0$, we have by Binet's second formula (see \cite{MR0178117}) that $$C(\sigma):= \max_{\alpha \ge 0} \frac{\sigma^{\alpha}}{\Gamma(\alpha)} < \infty.$$  Thus, for any $\sigma > 0$, we have that 
\begin{equation}\label{lbt_norm_eq}
	\|\bo F\|_{\sigma} = \sum \frac{(\sigma/r)^\alpha}{\Gamma(\alpha)} |a_\alpha| r^{\alpha} \le C(\sigma/r) \|F\|_{r}.
\end{equation}
Since $F$ has natural support, the sum is finite for all $\sigma$; so the series $\bo F(X)$ has infinite radius of convergence, and its log-sum $\bar{\bo F}$ is also log-holomorphic on $\bar\CC$.  In summary:

\begin{prop}\label{lbt_asymptotic_prop}
	Let $F(X)$ be a convergent generalized power series with natural support.  Then both $\bo\bar{F}$ and $\bar{\bo F}$ are log-holomorphic on $\bar\CC$, and we have $\bar{\bo F} = \bo\bar F$.
\end{prop}

\begin{proof}
	Since $F(X)$ has natural support, we write $F(X)= \sum_{n=0}^N a_n r^{\alpha_n}$ with either $N \in \NN$, or $N=\infty$ and $\lim_{n \to \infty} \alpha_n = +\infty$.  For $w \in \CC$, let $\bar S$ be the closure of the $\log$-sector $S(\im w,\log r,\pi)$, and define $K:\partial \bar S \into \CC$ by $$K(\eta):= \frac1{2\pi i} e^{w-\eta} e^{e^{w-\eta}}.$$  For $n \in \NN$ with $n \le N$, let $u_n: \partial \bar S \into \CC$ be defined by $$u_n(\eta):= a_n e^{\alpha_n \eta} K(\eta) = a_n p_{\alpha_n}(\eta) K(\eta),$$  where $p_{\alpha_n}$ is defined as in \cref{power_expl}.  Proceeding as in the proof of \cref{lbt_sup_estimate}, we obtain a $C>0$ such that $$\int_{\partial \bar S} |u_n(\eta)| d\eta \le C|a_n|r^{\alpha_n}, \quad\text{for each } n.$$  Since $\|F\|_r < \infty$, it follows that $\sum_n \int_{\partial \bar S} |u_n(\eta)| d\eta < \infty$.  If follows from analysis that the functions $u_n$, for each $n$, as well as $\sum_n u_n$ and $\sum_n |u_n|$ are integrable on $\partial \bar S$ and that
	\begin{align*}
		\bar{\bo F}(w) &= \sum_n \frac{a_n}{\Gamma(\alpha_n)} e^{\alpha_n w}= \sum_n a_n (\bo p_{\alpha_n})(w) \\
		&= \sum_n \int_{\partial \bar S} u_n(\eta) d\eta = \int_{\partial \bar S} \left(\sum_n u_n(\eta)\right) d\eta \\
		&= \int_{\partial \bar S} \bar F(\eta) K(\eta) d\eta = (\bo \bar F)(w),
	\end{align*}
	as claimed.
\end{proof}

\subsection{Generalized multisummable functions} \label{pre_gen_multi_section}

We now define generalized multisummable functions inspired by Tougeron's characterization of multisummable functions \cite{Tougeron:1994fk} and by their presentation in \cite{Dries:2000mx}.  However, while it was possible in \cite{Dries:2000mx} to refer to the existing literature for summability, it is not the case in our setting. More precisely, our aim is to show a quasianalyticity result for our functions analogous to that in \cite[Proposition 2.18]{Dries:2000mx}. To this end, we need to introduce suitable Borel and Laplace transforms adapted to the generalized multisummable framework (see \cref{ramified_section}). The presentation turns out to be more readable in this setting by replacing the usual ``Gevrey order" $k$ by $1/k$. This leads to the following definitions.

For $R,k \ge 0$, $\theta>\pi/2$ and $p \in \NN$, we set $$\rho^{R,k}_p:= \frac R{(1+p)^{k}} \quad\text{and}\quad S^{R,k}_p:= \cl\left(S\left(0,\log R, {\theta k}\right) \cup H\left(\log\rho^{R,k}_p\right)\right).$$ 

Let $K \subseteq [0,\infty)$ be a nonempty finite set and $r > 1$ (note that the situation studied  in \cite{Dries:2000mx} corresponds, in the current notation, to $\pi/2<\theta<\pi$ and $K\subseteq [0,1]$, in order to avoid dealing with the logarithmic chart), and set \begin{equation}\label{M_K}M_K:= \max K,\ \mu_K:= \min K.\end{equation}  Moreover, we fix a natural set $\Delta \subseteq [0,\infty)$ and set $\tau:= (K,R,r,\theta,\Delta)$ (note that $\Delta = \NN$ in \cite{Dries:2000mx}). We define $$S^\tau:= \bigcap_{k \in K \setminus \{0\}} S\left(0,\log R, {\theta k}\right) \text{ if } K \ne \{0\},\  \quad\text{and}\quad  S^\tau:= H(\log R) \text{ if } K=\{0\}$$ and, for $p \in \NN$, $$\rho^\tau_p:= \min_{k \in K}\rho^{R,k}_p = \rho^{R,M_K}_p$$ and  $$S^\tau_{p}:= \bigcap_{k \in K} S^{R,k}_p.$$

\begin{nrmk} \label{ms_domain_rmk}
	If $0 \in K$ and $K' := K \setminus\{0\}$ is nonempty, then $S^\tau = S^{\tau'}$ and $S^\tau_p = S^{\tau'}_p$ for all $p$, where $\tau' = (K',R,r,\theta,\Delta)$.
\end{nrmk}

\begin{df}\label{tougeron_sum}
	For each $p \in \NN$, let $f_p:S^\tau_p \into \CC$ be log-holomorphic, that is, there exists a log-domain $D_p \supseteq S^\tau_p$ and a log-holomorphic $g_p:D_p \into \CC$ such that $f_p = g_p\rest{S^\tau_p}$.  Moreover, we assume that there are  generalized power series $F_p(X) \in \Pc{C}{X^*}$ with support contained in $\Delta$ such that $\|F_p\|_{\rho^\tau_p} < \infty$ and $$f_p(w) = \bar{F_p}(w) \quad\text{for } w \in H(\log\rho^\tau_p).$$  Assume also that $$\sum_p \|F_p\|_{\rho^\tau_p}r^p < \infty \quad\text{and}\quad \sum_p \|f_p\|_{S^\tau_p} r^p < \infty,$$ where $\|f_p\|_{S^\tau_p}:= \sup_{w \in S^\tau_p} |f_p(w)|$ denotes the sup norm of $f_p$ on $S^\tau_p$.  The second of these finiteness assumptions implies that $\sum_p f_p$ converges uniformly on $S^\tau \setminus \{-\infty\}$ to a holomorphic function $g:S^\tau \setminus\{-\infty\} \into \CC$, while the first implies that this $g$ extends continuously to $-\infty$, so that the resulting $g:S^\tau \into \CC$ is log-holomorphic.  From now on, we abbreviate this situation by writing $$g =_\tau \sum_p f_p.$$
	Thus, for a log-holomorphic function $f:S^\tau \into \CC$, we set $$\|f\|_\tau:= \inf\set{\max\set{\sum_p \|F_p\|_{\rho^\tau_p}r^p,\sum_p \|f_p\|_{S^\tau_p} r^p}:\ f =_\tau \sum_p f_p} \in [0,\infty];$$  note that $\|f\|_\tau < \infty$ if and only if there exists a sequence $f_p$ such that $f =_\tau \sum_p f_p$.
\end{df}

We set $$\G_\tau:= \set{f: S^\tau \into \CC:\ f \text{ is log-holomorphic and } \|f\|_\tau < \infty}.$$
It is immediate from this definition that $\G_\tau$ is a $\CC$-vector space under pointwise addition; moreover, if $\Delta$ is closed under addition, then $\G_\tau$ is closed under multiplication of functions, making $\G_\tau$ a $\CC$-algebra.  

\begin{nprclaim}[Convention]
	If $\Delta$ is natural, then so is its closure under addition; so we assume from now on that $\Delta$ is closed under addition.
\end{nprclaim} 

\begin{expl} \label{classic_expl}
	Tougeron's characterization implies that, if $M_K < 2$ (where $M_K$ is as in \eqref{M_K}) and $f:S\left(0,R,{\theta M_K}\right) \into \CC$ is such that $f \circ \log$ is $K$-summable, then $f \in \G_\tau$, where $\tau = (K,R, r, \theta,\NN)$ for some $r > 1$. 
\end{expl}

\begin{dfs} \label{k-summable_def}
	\begin{enumerate}
		\item We call a function $f$ \textbf{generalized multisummable in the real direction} if $f \in \G_\tau$ for some $\tau$ as above. 
		\item We call a function $f$ \textbf{generalized $K$-summable in the real direction} if there exist $R'>0$, $r'>1$, $\theta' > \pi/2$ and a natural $\Delta' \subseteq [0,\infty)$ such that $f \in \G_{\tau'}$ with $\tau' = (K,R',r',\theta',\Delta')$. 
	\end{enumerate}
\end{dfs}

\begin{expl}\label{gevrey_2_to_gevrey_1}
	In terms of \cref{classic_expl}, Tougeron's characterization of multisum\-mable functions implies that if $f$ is $K$-summable in the positive real direction, then $f \circ \exp$ is generalized $K$-summable in the real direction.  
\end{expl}

Let $f \in \G_\tau$ with associated functions $f_p$ and series $F_p$ be as in \cref{tougeron_sum} such that $\sum_p \|F_p\|_{\rho^\tau_p}r^p < 2\|f\|_\tau$ and $\sum_p \|f_p\|_{S^\tau_p} r^p < 2\|f\|_\tau$.
In \cref{k-sum_asymp_lemma} below, we show that $f$ has asymptotic expansion $F(e^w)$ at $-\infty$, for the generalized power series $F(X)$ with support contained in $\Delta$ (and hence natural), defined in \eqref{F} below.  
To do so, say $F_p(X) = \sum a_{p,\alpha} X^\alpha$ for each $p$, where each $a_{p,\alpha} \in \RR$, and write $\rho_p$ for $\rho_p^\tau$.  Then for each $p$ and $\alpha$, and for arbitrary $s \in (1,r)$, we have $$|a_{p,\alpha}| \le \frac{\|F_p\|_{\rho_p}}{\rho_p^\alpha} = \frac{\|F_p\|_{\rho_p}}{R^\alpha} (p+1)^{\alpha M_K} \le \frac{\|F_p\|_{\rho_p}}{R^\alpha} s^p (p+1)^{\alpha M_K}.$$  Therefore, for each $\alpha$ and arbitrary $s \in (1,r)$, we get 
\begin{align*}
	\sum_p |a_{p,\alpha}| &\le \frac1{R^\alpha} \sum_p (p+1)^{\alpha M_K} \|F_p\|_{\rho_p} s^p \\
	&= \frac1{R^\alpha} \sum_p \|F_p\|_{\rho_p} r^p (p+1)^{\alpha M_K} \left(\frac sr\right)^p \\
	&\le \frac{C(s,\alpha)}{R^\alpha} \sum_p \|F_p\|_{\rho_p} r^p < \infty,
\end{align*}
where $C(s,\alpha):= \max_p (p+1)^{\alpha M_K} (s/r)^p < \infty$.  So we set $$a_\alpha := \sum_p a_{p,\alpha},$$ for each $\alpha$, and \begin{equation}\label{F}
	F(X) := \sum a_\alpha X^\alpha,
\end{equation} which has support contained in $\Delta$.

\begin{lemma} \label{p-remainder}
	There exist $D,E> 0$ such that for all $p \in \NN$ and all $\beta \ge 0$, we have $$\left|f_p(w) - \sum_{\alpha < \beta} a_{\alpha,p} e^{\alpha w}\right| \le CD^\beta  \frac{(p+1)^{\beta M_K}}{r^p} \left|e^{\beta w}\right| \quad\text{for } w \in S_p^\tau.$$
\end{lemma}

\begin{proof}
	Fix $p \in \NN$ and $\beta \ge 0$, and let $w \in S_p^\tau$.  We distinguish two cases:\medskip
	
	\noindent\textbf{Case 1:} $\re w < \log\rho_p$. Then 
	\begin{align*}
		\left|f_p(w) - \sum_{\alpha < \beta} a_{p,\alpha} e^{\alpha w}\right| &= \left|\sum_{\alpha \ge \beta} a_{p,\alpha} e^{\alpha w}\right| \\
		&\le \left|e^{\beta w}\right| \sum_{\alpha \ge \beta} |a_{p,\alpha}| \left|e^{(\alpha-\beta) w}\right| \\
		&\le \left|e^{\beta w}\right| \sum_{\alpha \ge \beta} |a_{p,\alpha}| (\rho_p)^{\alpha-\beta} \\
		&\le \frac{\left|e^{\beta w}\right|}{(\rho_p)^\beta} \|F_p\|_{\rho_p} \\
		&\le 2\left|e^{\beta w}\right| \frac{(p+1)^{\beta M_K}}{R^\beta r^p} \|f\|_\tau,
	\end{align*} which proves the estimate in this case. \medskip
	
	\noindent\textbf{Case 2:} $\re w \ge \log\rho_p$.  Then 
	\begin{equation*}
	\left|f_p(w) - \sum_{\alpha < \beta} a_{p,\alpha} e^{\alpha w}\right| \le \left|f_p(w)\right| + \sum_{\alpha < \beta} |a_{p,\alpha}|\left|e^{\alpha w}\right|,
	\end{equation*}
	so we further split up the estimate:
	\begin{align*}
		|f_p(w)| &\le \|f_p\|_{S_p^\tau} \\
		&\le \|f_p\|_{S_p^\tau} \frac{|e^{\beta w}|}{(\rho_p)^\beta} &\text{as } |e^{w}| \ge \rho_p \\
		&= \frac{(p+1)^{\beta M_K}}{R^\beta} \|f_p\|_{S_p^\tau} \left|e^{\beta w}\right| \\
		&\le \frac{(p+1)^{\beta M_K}}{R^\beta r^p} 2\|f\|_{\tau} \left|e^{\beta w}\right|,
	\end{align*}
	while
	\begin{align*}
		\sum_{\alpha < \beta} |a_{p,\alpha}| \left|e^{\alpha w}\right| &\le \left|e^{\beta w}\right| \sum_{\alpha < \beta} |a_{p,\alpha}| (\rho_p)^{\alpha-\beta} &\text{as } \alpha-\beta < 0 \text{ and } |e^{ w}| \ge \rho_p \\
		&= \frac{\left|e^{\beta w}\right|}{(\rho_p)^\beta} \sum_{\alpha < \beta} |a_{p,\alpha}|(\rho_p)^\alpha \\
		&\le \left|e^{\beta w}\right| \frac{(p+1)^{\beta M_K}}{R^\beta} \|F_p\|_{\rho_p} \\
		&\le \left|e^{\beta w}\right| \frac{(p+1)^{\beta M_K}}{R^\beta r^p} \cdot 2\|f\|_\tau.
	\end{align*}
	This completes the proof of Case 2 and therefore of the lemma.
\end{proof}

\begin{prop}[Gevrey estimates]\label{gevrey_estimates} \label{k-sum_asymp_lemma}
	For every closed log-subsector $\bar S_0$ of $S^\tau$, there exist $D,E>0$ such that, for each $\beta \ge 0$, $$\left|f(w) - \sum_{\alpha < \beta} a_\alpha e^{\alpha w}\right| \le DE^\beta \Gamma\left(\beta M_K\right) \left|e^{\beta w}\right| \quad\text{for } w \in \bar S_0.$$
\end{prop}

\begin{proof}
	Let $D,E>0$ be obtained from \cref{p-remainder}, and let $\beta \ge 0$.  Since $\Delta \cap [0,\beta)$ is finite we have, for $w \in \bar S_0$,
	\begin{align*}
		\left|f(w) - \sum_{\alpha < \beta} a_\alpha e^{\alpha w}\right| &= \left|\left(\sum_p f_p(w)\right) - \sum_{\alpha < \beta} \left(\sum_p a_{\alpha,p}\right) e^{\alpha w}\right| \\
		&\le \sum_p \left|f_p(w) - \sum_{\alpha < \beta} a_{\alpha,p} e^{\alpha w}\right| \\
		&\le CD^\beta  \left(\sum_p \frac{(p+1)^{\beta M_K}}{r^p}\right) \left|e^{\beta w}\right| &\text{by \cref{p-remainder}}.
	\end{align*}
	Since $\sum_p \frac{(p+1)^{\beta M_K}}{r^p} \le C'(D')^\beta \beta^{\beta M_K}$ for some $C',D'>0$  (see, for instance, the proof of \cite[Lemma 2.6]{Dries:2000mx}), the proposition now follows from Stirling's formula for $\Gamma$ (see \cite{MR0165148}).
\end{proof}

\cref{k-sum_asymp_lemma} implies that $F(e^w)$ is an asymptotic expansion of $f$ at $-\infty$; hence, it is uniquely determined by $f$ (and is, in particular, independent of the particular sequence $\{f_p\}$), and we write $Tf(X) := F(X)$.  The map $T:\G_\tau \into \Ps{C}{X^*}$ is a $\CC$-algebra homomorphism.

\begin{nrmk}\label{complete_rmk}
	Standard methods for proving topological completeness of function spaces (see e.g. Rudin's Real and Complex Analysis) show that the normed algebra $\left(\G_\tau,\|\cdot\|_\tau\right)$ is complete; we leave the details to the reader.
\end{nrmk}
	
\begin{expl} \label{convergent_expl}
	Assume that $K=\{0\}$. If $f \in \G_\tau$, then $Tf$ converges, $\|Tf\|_R < \infty$ and $f = \bar{Tf}$.  To see this, let $f =_\tau \sum_p f_p$ with $Tf_p = \sum_{\alpha} a_{p,\alpha} X^\alpha$ for each $p$; then $Tf = \sum_{\alpha} a_{\alpha}X^\alpha$ with $a_\alpha = \sum_p a_{p,\alpha}$ as above.  Since $\rho_p^\tau = R$ for each $p$, the assumption $\sum_p \|Tf_p\|_R r^p < \infty$ implies that $\sum_{p,\alpha} |a_{p,\alpha}| R^\alpha < \infty$.  This implies that the family $\{a_{p,\alpha} e^{\alpha w}\}$ is summable on $\cl H(\log R)$; in particular, the order of summation can be changed.  Thus, we have for $w \in \cl H(\log R)$ that
	\begin{equation*}
		f(w) = \sum_p f_p(w) = \sum_p \sum_\alpha a_{p,\alpha} e^{\alpha w} = \sum_\alpha \left(\sum_p a_{p,\alpha}\right) e^{\alpha w} = \sum_\alpha a_\alpha e^{\alpha w} = \bar{Tf}(w),
	\end{equation*}
	as claimed.
	
	Conversely, if $F \in \Ps{R}{X^*}$ has natural support and satisfies $\|F\|_R < \infty$, then (the appropriate restriction of) the function $\bar{F}$ belongs to $\G_{\tau}$, where $\tau = \left(K,R, r, \theta,\Delta\right)$ with $K$, $r$, $\theta$ and $\Delta \supseteq \supp(F)$ arbitrary.  To see this, simply take $f_0 := \bar F$ and $f_p:= 0$, for $p > 0$.
\end{expl}

\begin{nrmk} \label{ms_domain_rmk 2}
	Assume that $0 \in K$ and $|K|>1$, and set $K':= K \setminus\{0\}$ and $\tau':= (K',R,r,\theta,\Delta)$.  Then by \cref{ms_domain_rmk}, we have $\G_\tau = \G_{\tau'}$.
\end{nrmk}

	Recall from \ref{M_K} that $\mu_K= \min K \in [0,\infty)$.  If $\mu_K \ge 1$, then ${\theta \mu_K} > \frac\pi2$, so by \cref{Borel_prop}, the function $\bo f$ is defined on the sector $S\left(0, \infty, {\theta \mu_K} - \frac\pi2\right)$.  
	The next proposition explains in more detail what happens when we apply the Borel transform to $f$.

\begin{prop} \label{lbt_on_pre-summable}
	Let $f \in \G_\tau$, and assume that $\mu_K \ge 1$ (in particular, $\bo f$ is well defined).  Let $R' \in (0,R)$ and $r' \in (1,r)$ be such that $R' \le \frac Re \log(r/r')$, and set $K':= \set{k-1:\ k \in K}$ and $\tau':= \left(K',R',r',\theta,\Delta\right)$.  Then $\bo f$ belongs to $\G_{\tau'}$ and satisfies $T(\bo f) = \bo(Tf)$.
\end{prop}

\begin{proof}
	For $p \in \NN$, we set $\rho_p:= \rho_p^\tau$ and  $\sigma_p := \rho_p^{\tau'}$.  Then for $\alpha \ge 0$, we have $$\frac{(\sigma_p/\rho_p)^\alpha}{\Gamma(\alpha)} = \left(\frac {R'}R\right)^\alpha \frac{(p+1)^\alpha}{\Gamma(\alpha)}.$$  
	
	\noindent\textbf{Claim:} There exists $C = C(r,r',R,R') > 0$ such that $$\left(\frac {R'}R\right)^\alpha\frac{(p+1)^\alpha}{\Gamma(\alpha)}\left(\frac {r'}r\right)^p \le C,$$ for all $p \in \NN$ and $\alpha \ge 0$. \smallskip
	
	To see the claim, note that the function $x \mapsto f_\alpha(x):= (x+1)^\alpha \left(\frac {r'}r\right)^x$ attains its maximum at $x_\alpha = \frac\alpha{\log(r/r')}-1$; so this maximum is $$f_\alpha(x_\alpha) = \frac r{r'} A^\alpha \alpha^\alpha,$$ with $A:= \frac{(r'/r)^{1/\log(r/r')}}{\log(r/r')} = (e\log(r/r'))^{-1}$ independent of $\alpha$.  From Binet's second formula we get a constant $C''>0$ such that, for all $\alpha \ge 0$,
	\begin{equation*}
		\alpha^\alpha \le \frac{\sqrt{\alpha}e^\alpha}{C'' e^{\phi(\alpha)}}\Gamma(\alpha),
	\end{equation*}
	where $\phi(x)$ is the Stirling function.  Since the latter is bounded at $\infty$, there is a constant $C'>0$ such that 
	\begin{equation*}
		\alpha^\alpha \le C'\sqrt{\alpha}e^\alpha\Gamma(\alpha) \le C' e^{2\alpha} \Gamma(\alpha).
	\end{equation*}
	Therefore,
	\begin{equation*}
		\left(\frac {R'}R\right)^\alpha (p+1)^\alpha\left(\frac {r'}r\right)^p \le \frac r{r'} \left(\frac {R'}R\right)^\alpha A^\alpha \alpha^\alpha \le C' \frac r{r'} \left(\frac {R'}R K e^2\right)^\alpha \Gamma(\alpha) \le C \Gamma(\alpha),
	\end{equation*}
	where $C:= C'\frac r{r'}$, because our assumptions on $r'$ and $R'$ imply that $Ae^2 \le \frac R{R'}$.  This proves the claim.
	\smallskip
	
	It follows from the claim that 
	\begin{equation}\label{sup_norm_eq}
		\|\bo F_p\|_{\sigma_p}\left(\frac {r'}r\right)^p = \sum \frac{(\sigma_p/\rho_p)^\alpha}{\Gamma(\alpha)}\left(\frac {r'}r\right)^p |a_{p,\alpha}|\rho_p^\alpha \le C\|F_p\|_{\rho_p}
	\end{equation} for each $p \in \NN$, so that 
	\begin{equation} \label{first_norm_ineq}
		\sum_p \|\bo F_p\|_{\sigma_p} (r')^p \le C \sum_p \|F_p\|_{\rho_p} r^p < \infty.
	\end{equation} 
		
	Next, we show $\sum \|\bo f_p\|_{S^{\tau'}_p} (r')^p < \infty$.  If $|K| = \mu_K = 1$, then \eqref{first_norm_ineq} also proves that $\bo F$ is convergent and that $\sum \|\bo f_p\|_{H(\log R')} (r')^p < \infty$; this settles the assertion in this case.
	 
	So assume that $\mu_K > 1$ or $|K| > 1$.  By \eqref{first_norm_ineq}, it suffices to show that $\sum \|\bo f_p\|_{S^{\tau'}} (r')^p < \infty$ and, for $k \in K\setminus\{M_K\}$, that $\sum \|\bo f_p\|_{S^k_p} (r')^p < \infty$, where we set  $$\sigma^k_p:= \rho_p^{R',k-1} \quad\text{and}\quad S^k_p := S\left(0, \log\sigma^k_p, {\theta(k-1)}\right).$$
	
	For the first estimate: since ${\theta(\mu_K-1)}< {\theta \mu_K} - \frac\pi2$, \cref{lbt_sup_estimate} shows that $\left\|\bo f_p\right\|_{S^{\tau'}} \le  C\left\|f_p\right\|_{S^\tau}$, for some constant $C>0$ that only depends on $\theta$ and $\theta'$.
	
	For the second estimate: fix $k \in K\setminus\{M_K\}$ and set $\rho^k_p:= \rho_p^{R,k}$, $k' := \min(K \setminus [\mu_K,k])$ and $T^k_p:= S\left(0,\log\rho^k_p,{\theta k'}\right)$.  Since ${\theta(k'-1)} < {\theta k'} - \frac\pi2$, Lemma 1.3 shows that $$\left\|\bo f_p\right\|_{S^k_p} \le C \left\|f_p\right\|_{T^k_p} e^{\sigma^k_p / \rho^k_p} \frac{\sigma^k_p}{\rho^k_p} = C\|f_p\|_{T^k_p} \left(e^{R'/R}\right)^{1+p} \frac{R'}R (1+p),$$ for some constant $C>0$ independent of $p$.  
	Now note that our assumptions on $r'$ and $R'$ imply that $e^{R'/R}\cdot \frac{r'}r < 1$; therefore, $D:= \max_p (1+p) \left(e^{R'/R}\right)^{(1+p)}\left(\frac {r'}r\right)^p < \infty$.  It follows that
	\begin{equation*}
	\sum_p \|\bo f_p\|_{S^k_p} (r')^p \le D \sum_p \left\|f_p\right\|_{T^k_p} r^p < \infty,
	\end{equation*}
	as claimed.

	Finally, it remains to show that $\bo f =_{\tau'} \sum_p \bo f_p$ and $T(\bo f) = \bo(Tf)$: set $g:= \sum_p \bo f_p$.  The above estimates show that $g \in \G_{\tau'}$ with $Tg = \sum_p \bo F_p = \bo(Tf)$, so we need to show that $\bo f = g$.  Since $f = \sum_p f_p$ uniformly in $S^\tau$, it follows from integration theory that $\bo f = \bo\left(\sum_p f_p\right) = \sum_p \bo f_p = g$, as required.
\end{proof}

For later purposes, we note the following special case of \cref{lbt_on_pre-summable}:

\begin{cor} \label{multisum_cor}
	Let $f \in \G_\tau$, and assume that $|K| \ge 1$ and $\mu_K = 1$.  Let $R' \in (0,R)$ and $r' \in (1,r)$ be such that $R' \le \frac Re \log(r/r')$, and set $K':= \set{k-1:\ k \in K}$ and $\tau':= \left(K',R',r',\theta,\Delta\right)$.  Then $\bo f$ belongs to $\G_{\tau'}$ and satisfies $T(\bo f) = \bo(Tf)$. 
\end{cor}

\begin{proof}
	Since $\mu_K = 1$, we have $\mu_K-1 = 0$.  So the corollary follows from \cref{lbt_on_pre-summable} and \cref{ms_domain_rmk 2}.
\end{proof}

\subsection{Ramified logarithmic transforms} \label{ramified_section}

Let $\lambda>0$.  In the classical situation, the ramified Borel transform $\bo^\lambda f$ of a function $f$ is obtained from $\bo f$ by the change of variables $z \mapsto z^\lambda$.  In the logarithmic chart, this means that $\bo^\lambda f$ is obtained from $\bo f$ by the change of variables $w \mapsto \lambda w$.  

\subsubsection{Ramified logarithmic Borel transform} \label{k-lbt}
Let $d,\frak r \in \RR$ and $\theta > \pi$, and write $S = S(d,\frak r,\theta\lambda)$.  We denote by $\mm_\lambda:\CC \into \CC$ the \textbf{logarithmic ramification} map defined by $\mm_\lambda(w):= \lambda w$.

Let $f:S \into \CC$ be such that $f\rest{\tilde S}$ is bounded and log-holomorphic, for every closed log-subsector $\tilde S$ of $S$.  Then the map $f \circ \mm_{\lambda}: S(d/\lambda,\frak r/\lambda,\theta) \into \CC$ has logarithmic Borel transform $\bo(f \circ \mm_{\lambda}):S\left(d/\lambda,\infty,\theta-\frac\pi2\right) \into \CC$.  We define the \textbf{log-$\lambda$-Borel transform} $\bo^\lambda f: S\left(d,\infty,{\theta\lambda} - \frac{\pi\lambda}{2}\right) \into \CC$ of $f$ by
\begin{equation*}
\bo^\lambda f := \left(\bo (f \circ \mm_{\lambda})\right) \circ \mm_{1/\lambda}.
\end{equation*}


We immediately obtain the following from \cref{Borel_prop} and \cref{power_expl}.

\begin{cor}\label{k-Borel_cor}
	Let $f:S \into \CC$ be such that $f\rest{\tilde S}$ is bounded and log-holomorphic, for every closed log-subsector $\tilde S$ of $S$, and set $S':= S\left(d,\infty,{\theta\lambda} -\frac{\pi\lambda}{2}\right)$.
	\begin{enumerate}
		\item For every closed log-subsector $\tilde S$ of $S'$, the function $(\bo^\lambda f)\rest{\tilde S}$ is log-holomorphic and there exist $C,M>0$ such that $$|(\bo^\lambda f)(w)| \le C e^{De^{(\re w)/\lambda}} \quad\text{for } w \in \tilde S.$$
		\item Let $\alpha \ge 0$, and assume that for every closed log-subsector $\tilde S$ of $S$, we have $|f(w)| = O\left(e^{\alpha\re w}\right)$ as $w \to -\infty$ in $\tilde S$.  Then, for every closed log-subsector $\tilde S$ of $S'$, we have $|(\bo^\lambda f)(w)| = O\left(e^{\alpha\re w}\right)$ as $w \to -\infty$ in $\tilde S$.
		\item For $\alpha \ge 0$, we have $\bo^\lambda p_\alpha =  \frac{p_\alpha}{\Gamma(\alpha\lambda)}$.\qed
	\end{enumerate}
\end{cor}

\subsubsection{Ramified logarithmic Laplace transform} \label{k-llt}

Let $\varphi > 0$ and set $S:= S(d,\infty,\varphi\lambda)$.  Let $f:S \into \CC$ be a function, and assume that for every closed log-subsector $\tilde S$ of $S$, the restriction $f\rest{\tilde S}$ is log-holomorphic and there exist $C,D>0$ such that $|f(w)| \le Ce^{De^{(\re w)/\lambda}}$ for $w \in \tilde S$.  

Let also $\theta \in (0,\varphi)$; then the map $f \circ \mm_{\lambda}:S\left(d/\lambda,\infty,\theta\right) \into \CC$ has logarithmic Laplace transform $\la(f \circ \mm_{\lambda}): S\left(d/\lambda,\frak r/\lambda,\theta + \frac\pi2\right) \into \CC$, for some $\frak r \le \log(D)\lambda$.  We define the \textbf{log-$\lambda$-Laplace transform} $\la^\lambda f:S\left(d,\frak r,{\theta\lambda} + \frac{\pi\lambda}{2}\right) \into \CC$ of $f$ by
$$\la^\lambda f:= \left(\la(f \circ \mm_{\lambda})\right) \circ \mm_{1/\lambda}.$$  We immediately obtain the following from \cref{Laplace_prop} and \cref{power_expl}.

\begin{cor}\label{k-Laplace_cor}
	Let $f:S \into \CC$ be a function, and assume that for every closed log-subsector $\tilde S$ of $S$, the restriction $f\rest{\tilde S}$ is log-holomorphic and there exist $C,D>0$ such that $|f(w)| \le Ce^{De^{(\re w)/\lambda}}$ for $w \in \tilde S$.  Let also $\theta \in (0,\varphi)$, and let $\frak r \le \log(D)\lambda$ be as above and set $S':= S\left(d,\frak r,{\theta\lambda} + \frac{\pi\lambda}{2}\right)$.
	\begin{enumerate}
		\item Let $\alpha \ge 0$, and assume that for every closed log-subsector $\tilde S$ contained in $S$, we have $|f(w)| = O\left(e^{\alpha\re w}\right)$ as $w \to -\infty$ in $\tilde S$.  Then, for every closed log-subsector $\tilde S$ contained in $S'$, we have $|(\la^\lambda f)(w)| = O\left(e^{\alpha\re w}\right)$ as $w \to -\infty$ in $\tilde S$.
		\item For $\alpha \ge 0$, we have $\la^\lambda p_\alpha =  \Gamma(\alpha\lambda)p_\alpha$.\qed
	\end{enumerate}
\end{cor}

For $f:S(d,\frak r,\theta\lambda) \into \CC$ as in \cref{k-lbt}, \cref{k-Borel_cor} implies that $\la^\lambda(\bo^\lambda f)$ is defined and log-holomorphic on every closed log-subsector $\tilde S$ contained in $S(d,\frak r,\theta\lambda)$.  From \cref{llt-lbt_prop}, we therefore obtain:

\begin{cor} \label{k-llt-lbt_prop}
	For $f:S(d,\frak r,\theta\lambda) \into \CC$ as in \cref{k-lbt}, we have $\la^\lambda(\bo^\lambda f) = f$. \qed
\end{cor}

\subsubsection{Formal Borel and Laplace transforms}
Let now $F(X) = \sum a_\alpha X^\alpha \in \Pc{C}{X^*}$.  In view of the above, we define the \textbf{formal $\lambda$-Borel transform} $$(\bo^\lambda F)(X) := \sum \frac{a_\alpha}{\Gamma(\alpha\lambda)} X^\alpha$$ and the \textbf{formal $\lambda$-Laplace transform} $$(\la^\lambda F)(X):= \sum \Gamma(\alpha\lambda) a_\alpha X^\alpha.$$  We get the following from  \cref{lbt_asymptotic_prop}:

\begin{cor} \label{k-lbt_conv}
	Let $F$ be a convergent generalized power series with natural support.  Then both $\bo^\lambda\bar F$ and $\bar{\bo^\lambda F}$ are log-holomorphic on $\bar\CC$, and we have $\bo^\lambda \bar F = \bar{\bo^\lambda F}$. \qed
\end{cor}

\subsubsection{Ramified Borel transforms of generalized multisummable functions}
Let $K \subseteq (0,\infty)$ be finite and nonempty and $\lambda \le \mu_K$, and define $$K(\lambda):= \set{k\lambda:\ k \in K} \quad\text{and}\quad \tau(\lambda):= \left(K(\lambda),R^{1/\lambda}, r, \theta \right);$$  note that $f \in \G_\tau$ if and only if $f \circ \mm_{\lambda} \in \G_{\tau(\lambda)}$.

Let $f \in \G_\tau$; by \cref{k-Borel_cor}, taking $\theta$ there equal to $\theta \cdot \frac{\mu_K}\lambda$ here, the function $\bo^\lambda f$ is defined on the log-sector $S\left(0, \infty, {\theta \mu_K} - \frac{\pi\lambda}{2}\right)$.  Thus, we obtain the following from  \cref{lbt_on_pre-summable} and \cref{multisum_cor}:

\begin{cor} \label{k-lbt_on_pre-summable}
	Let $f \in \G_\tau$, and assume that $\mu_K \ge \lambda$ (in particular, $\bo^\lambda f$ is well defined).  Let $R' \in (0,R)$ and $r' \in (1,r)$ be such that $(R')^{1/\lambda} \le \frac {R^{1/\lambda}} e \log(r/r')$, and set $$K':= \set{k-\lambda:\ k \in K,\ k > \lambda} $$ and $\tau':= \left(K',R',r',\theta,\Delta\right)$.  Then $\bo^\lambda f$ belongs to $\G_{\tau'}$ and satisfies $T(\bo^\lambda f) = \bo^\lambda(Tf)$. \qed
\end{cor}

\subsection{Summation and quasianalyticity} \label{quasi}
Let $K \subseteq [0,\infty)$ be nonempty and finite, and let $R>0$, $r > 1$, $\theta > \pi/2$, and $\Delta \subseteq [0,\infty)$ be natural and closed under addition, and set $\tau= (K,R,r,\theta,\Delta)$.  The goal of this section is to establish the quasianalyticity of the algebra $\G_\tau$ (\cref{quasianalyticity} below).  The key ingredient is the following summation method:

\begin{prop}[Summation] \label{watson_cor}
	Assume that $K = \{k_1, \dots, k_l\}$ with $0 < k_1 < \cdots < k_l < \infty$ and $l \ge 1$.  Let $f \in \G_\tau$, and set $\kappa_1:= k_1$ and $\kappa_i:= k_i - k_{i-1}$ for $i=2, \dots, l$.  Then the series $\left(\bo^{\kappa_1} \circ \cdots \circ \bo^{\kappa_{l}}\right)(Tf)$ converges, and we have $$f = \left(\la^{\kappa_1} \circ \cdots \circ \la^{\kappa_l}\right) \left(\bar{\left(\bo^{\kappa_l} \circ \cdots \circ  \bo^{\kappa_1}\right) (Tf)}\right).$$  
\end{prop}

\begin{proof}
	By induction on $l$.  If $l=1$, then 
	\begin{align*}
		\la^{\kappa_1} \left(\bar{\bo^{\kappa_1}(Tf)}\right) &= \la^{\kappa_1} \left(\bar{T(\bo^{\kappa_1} f)}\right) &\text{by \cref{k-lbt_on_pre-summable}} \\
		&= \la^{\kappa_1} \left(\bo^{\kappa_1} f\right) &\text{by \cref{convergent_expl}} \\
		&= f &\text{by \cref{k-llt-lbt_prop}}.
	\end{align*}
	
	So we assume that $l>1$ and the proposition holds for lower values of $l$.  Then by \cref{k-lbt_on_pre-summable}, the function $\bo^{\kappa_1} f$ belongs to $\G_{\tau'}$ and satisfies $T\left(\bo^{\kappa_1} f\right) = \bo^{\kappa_1}(Tf)$, where $\tau' = (K',R',r',\theta,\Delta)$ for $K' := \left(k_2-k_1, \dots, k_l-k_1\right)$ and some appropriate $R'>0$ and $r'>1$.  From the inductive hypothesis applied to $\bo^{k_1} f$, we get that $\left(\bo^{\kappa_l} \circ \cdots \circ \bo^{\kappa_{2}}\right) \left(T\left(\bo^{\kappa_1}f\right)\right)$ converges, so that
	\begin{align*}
	f &= \la^{\kappa_1}(\bo^{\kappa_1} f) &\text{(\cref{llt-lbt_prop})}\\
	&= \la^{\kappa_1}\left[\left(\la^{\kappa_{2}} \circ \cdots \circ \la^{\kappa_l}\right) \left(\bar{\left(\bo^{\kappa_l} \circ \cdots \circ \bo^{\kappa_{2}}\right) (T(\bo^{\kappa_1}f))}\right)\right] &\text{(ind. case applied to $\bo^{\kappa_1} f$)} \\
	&= \la^{\kappa_1}\left[\left(\la^{\kappa_{2}} \circ \cdots \circ \la^{\kappa_l}\right)\left(\bar{\left(\bo^{\kappa_l} \circ \cdots \circ \bo^{\kappa_{2}} \circ \bo^{\kappa_1}\right)(Tf)}\right)\right] \\
	&= \left(\la^{\kappa_1} \circ \la^{\kappa_{2}} \circ \cdots \circ \la^{\kappa_l}\right) \left(\bar{\left(\bo^{\kappa_l} \circ \cdots \circ \bo^{\kappa_{2}}\circ  \bo^{\kappa_1}\right) (Tf)}\right),
	\end{align*}
	as claimed.
\end{proof}

\begin{thm}[Quasianalyticity] \label{quasianalyticity}
	The map $T:\G_\tau \into \Ps{C}{X^*}$ is injective.
\end{thm}

\begin{proof}
	By \cref{convergent_expl} and \cref{ms_domain_rmk 2}, we may assume that $K \subseteq (0,\infty)$; so the theorem follows from \cref{watson_cor}.
\end{proof}

\subsection{Open questions} \label{question_1}
\begin{enumerate}
	\item Let $f$ be generalized $K$-summable in the real direction, where $K = (k_1, \dots, k_l)$.  Do there exist generalized $(k_i)$-summable functions $g_i$, for $i=1, \dots, l$, such that $f = g_1 + \cdots + g_l$?
	
	This question is motivated by the following: let $f$ be a $K$-summable function in the positive real direction (in the classical sense, at the origin; to avoid branching, let's assume $k_1 > \frac12$).  Then by  \cref{gevrey_2_to_gevrey_1}, the function $f \circ \exp$ belongs to $\G_\tau$ for some $\tau = (K,R,r,\theta,\NN)$; indeed, Tougeron's characterization implies that $f \circ \exp =_{\tau} \sum_p f_p \circ \exp$ for functions $f_p$ that are holomorphic at the origin.  This property of being holomorphic at the origin can be used, via Cauchy integration, to show that there exist $(k_i)$-summable functions $g_i$, for $i=1, \dots, l$, such that $f = g_1 + \cdots + g_l$.  However, for general $g \in \G_\tau$ with $g =_{\tau} \sum_p g_p$, the functions $g_p \circ \log$ have essential singularities at the origin, so the Cauchy integration argument used for $f$ does not work for $g \circ \log$ to write $g$ as a sum of a generalized $(k_i)$-summable functions.
	
	\item From the point of view of multisummability, there is nothing special about the real direction chosen here.  (We are only interested in the real direction here, because we are aiming to construct algebras of real functions.)  Indeed, one can similarly define generalized $K$-summable functions in any direction $d$.  However, it is not clear to us what the right generalization of ``generalized multisummable'' (without specified direction) should be: in the classical case, all multisummable functions are $2\pi i$-periodic in the logarithmic chart, and they are defined to be multisummable if they are multisummable in all but finitely many directions in $\RR/2\pi \ZZ$.  In contrast, the logarithmic sums of generalized convergent power series are not $ai$-periodic for any $a>0$ in general (take, for instance, the function $e^{\alpha w} + e^{\beta w}$ with $\alpha$ and $\beta$ linearly independent over $\QQ$), so generalized multisummable functions in the real direction aren't either.  Possibly, the right way to define ``generalized multisummable'' would be to look for something like Stokes phenomena in differential equations over (quotients of) convergent generalized power series.  
	
	\item Is there a Ramis-Sibuya theorem (see \cite{MR1303885}) for ordinary differential equations involving quotients of convergent generalized power series?  As hinted at in Question 2, such a theorem might inform the correct definition of the term ``generalized multisummable''.
\end{enumerate}

\section{Generalized multisummable functions in several variables} \label{several-variable_chapter}

The extension of the notion ``generalized multisummable in the real direction'' to several variables roughly follows the treatment in \cite[Section 2]{Dries:2000mx} of the notion ``multisummable in the positive real direction'', keeping in mind that we work in the logarithmic chart.  
Since we will not need to work with the ramified Borel and Laplace operators any more, we will revert here to the classical notation for Gevrey orders used in \cite{Dries:2000mx}.

It will be useful for the definitions below to set $\im(-\infty) = \arg 0 = 0$. 

\begin{prclaim}[Notation] \rm
	For $k = (k_1, \mdots,  k_m) \in [0, \infty)^m$ and $w = (w_1, \mdots,
	w_m) \in \bar\CC^m$ we put
	\begin{gather*}
	\Sigma k := k_1 + \cdots + k_m, \\
	\re w:= (\re w_1, \dots, \re w_m) \text{ and } \im w:= (\im w_1, \dots, \im w_m), \\
	kw:= (k_1w_1, \dots, k_mw_m) \text{ and } k \cdot w := k_1w_1 + \cdots + k_mw_m, \\
	|w| := \sup\set{|w_i|:\ i=1,\mdots, m} \text{ and } \|w\|:= (|w_1|, \dots, |w_m|), \\
	e^w := (e^{w_1}, \dots, e^{w_m}).
	\end{gather*}
	Moreover, if $z = (z_1, \dots, z_m) \in \CC^m$ is such that $\arg z_i \in (-\pi, \pi)$ for each $i$, we also set $$\log z:= (\log z_1, \dots, \log z_m),$$ where $\log$ denotes the standard branch of the logarithm.  Finally, if $\alpha \in [0,\infty)^m$ and $r \in (0,\infty)^m$, we put $$r^\alpha := r_1^{\alpha_1} \cdots r_m^{\alpha_m} \quad\text{and}\quad \Gamma(\alpha):= \Gamma(\alpha_1) \cdots \Gamma(\alpha_m).$$

	If $X \subseteq \bar\CC^m$ and $1\le \nu < m$, and if $a \in \CC^{\nu}$ and $b \in \CC^{m-\nu}$, then we let $$X_a:= \set{w \in \bar\CC^{m-\nu}:\ (a,z) \in X} \quad\text{and}\quad X^b:= \set{w \in \bar\CC^{\nu}:\ (z,b) \in X}$$ be the \textbf{fibers} of $X$ over $a$ and $b$, respectively.
	
	
	If $\frak r, \tilde{\frak r} \in \RR^m$,
	we write $\frak r \leq \tilde{\frak r}$ if $\frak r_i \leq \tilde{\frak r_i}$ for each
	$i$ (and similarly with ``$<$'' in place of ``$\leq$'').
\end{prclaim}

\subsection{Convergent generalized power series} \label{multi_cgps}
Let $X = (X_1, \dots, X_m)$.  Similar to the one-variable case, we denote by $\Ps{C}{X^*}$ the set of all \textbf{generalized power series} of the form $F(X) = \sum_{\alpha \in [0,\infty)^m} a_\alpha X^\alpha$, where each $a_\alpha \in \CC$ and the \textbf{support} $$\supp(F):= \set{\alpha \in [0,\infty)^m:\  a_\alpha \ne 0}$$ is contained in a cartesian product of  well-ordered subsets of $[0,\infty)$ (see \cite[Section 4]{Dries:1998xr} for details).  The series $F(X)$ \textbf{converges} if there exists a \textbf{polyradius} $r \in (0,\infty)^m$ such that $\|F\|_r:= \sum_\alpha |a_\alpha|r^\alpha < \infty$; we denote by $\Pc{C}{X^*}$ the set of all convergent generalized power series \cite[Section 5]{Dries:1998xr}.

For $\frak r \in \RR^m$, we let $$H(\frak r) := \set{w \in \bar\CC^m:\ \re w< \frak r} = H(\frak r_1) \times \cdots \times H(\frak r_m)$$ be the \textbf{log-disk} of \textbf{log-polyradius} $\frak r$.  For an set $D \subseteq \bar\CC^m$, we set $$D^\infty:= D \setminus \CC^m.$$  We call set $D \subseteq \bar\CC^m$ a \textbf{log-domain} if $D \cap \CC^m$ is a domain.  If $D \subseteq \bar\CC^m$ is a log-domain, a \textbf{log-holomorphic} function on $D$ is a continuous function $f:D \into \CC$ such that the restriction of $f$ to $D \cap \CC^m$ is holomoprhic.

Let $F(X) = \sum_{\alpha \in [0,\infty)^m} a_\alpha X^\alpha \in \Ps{C}{X^*}$ such that $\|F\|_r < \infty$.  Similar to  \cref{lbt_convergent}, there is a log-holomorphic function $\bar F: H(\log r) \into \CC$, called the \textbf{log-sum of $F$}, such that $\bar F(w) := F(e^w)$ whenever $|\im w| < \pi/2$.

\subsection{Logarithmic polydomains} \label{poly_section}
 We define here the logarithmic versions of the domains discussed in \cite[Section 2]{Dries:2000mx}.
Let $\frak r \in \RR^m$, $\theta >\pi/2 $ and $k \in
[0,\infty)^m$, and put
\begin{align*}
S^k(\frak r, \theta) &:= \set{w \in H(\frak r):\ k \cdot \|\im w\| <
	\theta} &\text{ (\bf log-$k$-polysector\rm)}.
\end{align*}
Note that if $m=1$, then $S^k(\frak r, \theta) = S(0,\frak r, \theta/k)$, where the latter is the sector defined in  \cref{prelim1}.
The reason for allowing $k_i=0$ is that we need our class of
log-polysectors to be closed under taking cartesian products with log-disks;
for instance, if $m>1$ and $k = (k', 0)$ with $k' \in [0,\infty)^{m-1}$, then
$S^k(\frak r, \theta) = S^{k'}(\frak r', \theta) \times H(\frak r_m)$.  Finally, our polystrips are ``in the real multidirection''; they can easily be defined in any multidirection\footnote{We would change our convention $\im(-\infty) = 0$, for a given multidirection $d \in \RR^m$, to $\im(-\infty_i) = d_i$ for each $i$, where $\infty_i$ denotes the logarithmic origin in the $i$-th coordinate.}, but we shall not do this here as we do not need to for our purposes.

Next, for $p \in \NN$ we put, by adapting \cite[Section 2]{Dries:2000mx} to the logarithmic chart, 
\begin{align*}
H_p^k(\frak r) &:= \set{w \in H(\frak r):\ k \cdot \re w < k \cdot \frak r - \log(1+p))}  &\text{(\textbf {log-$k$-polydisk}),} \\
S_p^{k}(\frak r, \theta) &:= S^k(\frak r, \theta) \cup H_p^k(\frak r).
\end{align*}
Note that $$\exp\left(H^k_p(\frak r)\right) = D^k_p\left(e^{\frak r}\right) = \set{z \in D\left(e^{\frak r}\right):\ \|z\|^k < \frac{e^{k\cdot \frak r}}{1+p}},$$  corresponding to \cite[Definition 2.1]{Dries:2000mx}.  Thus,
for a non-empty finite $K \subseteq [0,\infty)^m$, we set $$S^K(\frak r,\theta):= \bigcap_{k \in K} S^k(\frak r,\theta),$$ and for $p \in \NN$, $$S^K_p(\frak r,\theta):= \bigcap_{k \in K} S^k_p(\frak r,\theta).$$

Note that if $z \in S_p^K(\frak r, \theta)$ and $t \le 0$, then $t+z \in
S_p^K(\frak r, \theta)$; in particular, $S_p^K(\frak r, \theta)$ is connected.

\subsection{Generalized multisummable functions}  \label{multi_gms}

We now fix $R \in (0,\infty)^m$, $r>1$, $\theta > \pi/2$, a non-empty finite $K \subseteq [0,\infty)^m$ and a natural $\Delta \subseteq [0,\infty)^m$ that is closed under addition, and we set $\tau:= (K,R,r,\theta,\Delta)$.  In this situation, we write $$S^\tau:= S^K(\log R,\theta) \quad\text{and}\quad S^\tau_p:= S^K_p(\log R,\theta).$$    
\\

We need to introduce the following norms for generalized power series:  
let $U \subseteq \CC^m$ be an open neighbourhood of the origin such that $|z| \in U$ for every $z \in U$.
For a generalized power series $F \in \Pc{C}{X^*}$, we set $$\|F\|_{U} := \sup\set{\|F\|_s:\ s \in \cl(U) \cap (0,\infty)^n}.$$  It follows from the previous section that, if $\|F\|_{U} < \infty$, then $F$ is convergent and the log-sum of $F$ extends to a log-holomorphic function $\bar F: U \into \CC$ such that $\left\|\bar F\right\|_U \le \|F\|_U$, where $\left\|\bar F\right\|_U$ denotes the sup norm of $\bar F$ on $U$. 
\\

Similar to  \cref{pre_gen_multi_section}, we now define generalized multisummable functions in several variables.  
The role of the usual norm $\|\cdot\|_\rho$ on generalized power series there is taken on here by the norm $\|\cdot\|_U$ as defined above, where $U = D^k_p\left(R\right)$; note indeed  that $z \in D^k_p\left(R\right)$ implies $|z| \in D^k_p\left(R\right)$, as required.  Below, we denote this particular norm $\|\cdot\|_U$ by $\|\cdot\|_{R,k,p}$.

Thus, let $f_p:S^\tau_p \into \CC$ be log-holomorphic and bounded, and assume that there is a natural set $\Delta \subseteq [0,\infty)^m$ and, for each $p$, a convergent generalized power series $T(f_p)(X) \in \Pc{C}{X^*}$ with support contained in $\Delta$ such that $\|T(f_p)\|_{R,k,p} < \infty$ and $$f_p(w) = \bar{ T(f_p)}(w) \quad\text{for } w \in H^k_p\left(\log R\right).$$  Assuming that $$\sum_p \|T(f_p)\|_{R,k,p}\ r^p < \infty \quad\text{and}\quad \sum_p \|f_p\|_{S^\tau_p}\ r^p < \infty,$$ we have that $\sum_p f_p$ converges uniformly on $S^\tau$ to a log-holomorphic function $f:S^\tau \into \CC$.  As before, we abbreviate this situation by writing $$f =_\tau \sum_p f_p,$$  and we set $$\|f\|_\tau:= \inf\set{\max\set{\sum_p \|F_p\|_{R,k,p}\ r^p,\ \sum_p \|f_p\|_{S^\tau_p}\ r^p}:\ f =_\tau \sum_p f_p}.$$
Thus, as in \cref{pre_gen_multi_section}, we obtain a Banach algebra $$\G_\tau:= \set{f: S^\tau \into \CC:\ \text{there exist } f_p \text{ such that } f =_\tau \sum_p f_p}$$
over $\CC$ of functions in $m$ variables with norm $\|\cdot\|_\tau$. 
 Note that if $m=1$, $\G_\tau$ as defined here is the same as $\G_{\tau'}$ defined in  \cref{pre_gen_multi_section}, where $\tau'=(K',R,r,\theta,\Delta)$ with $K'$  obtained from $K$ by replacing all nonzero $k\in K$ by $1/k$ (see the introductory remarks at the beginning of \cref{several-variable_chapter}).

Arguing as in the one-variable case, if $T(f_p)(X) = \sum a_{\alpha,p} X^\alpha$ for each $p$, we obtain a generalized power series $Tf(X) = \sum a_\alpha X^\alpha$, where $a_\alpha = \sum_p a_{\alpha,p}$ for each $\alpha$.  
\subsection{Quasianalyticity}
In view of proving our o-minimality result in Section \ref{o-min_chapter}, we show in this section that the map $T:\G_\tau \into \Ps{C}{X^*}$ is injective.
First, we explain the reason why we define the log-$k$-polysectors and the log-$k$-polydisks via scalar products rather than just taking Cartesian products of the corresponding objects in one variable. This is done because scalar product behaves well with respect to fibers, as shown in Remark \ref{multi-dom_rmk}.

We set $\mu_{K,i} := \min\set{k_i:\ k \in K}$ for $i=1, \dots, m$, $\mu_K:= (\mu_{K,1}, \dots, \mu_{K,m})$ and $$\rho^\tau_p := \left(\frac{R_1}{(1+p)^{1/\mu_{K,1}}}, \dots, \frac{R_m}{(1+p)^{1/\mu_{K,m}}}\right).$$ 

The next lemma shows that, although the log-$k$-polydisks are not themselves log-disks, the set $S^\tau_p$ contains a suitable log-disk.

\begin{lemma}  \label{containment}
	We have $H\left(\log\rho^\tau_p\right) \subseteq H^{\mu_K}_p(\log R) \subseteq \bigcap_{k \in K} H^k_p(\log R) \subseteq  S^\tau_p$.
\end{lemma}   

\begin{proof}
	The first and third inclusions are straightforward.  For the second, let $k,l \in (0,\infty)^m$ be such that $l \le k$, $p \in \NN$ and $\rho \in \RR$; it suffices to show that $H^k_p(\rho) \subseteq H^k_p(\rho)$.
	To see this, let $w \in H^k_p(\rho)$.  Then $$k \cdot \re w = l \cdot \re w + (k-l) \cdot \re w < l \cdot \rho + (k-l) \cdot \re w - \log(1+p).$$  Since $\re w < \rho$ as well, it follows that $k \cdot \re w < l \cdot \rho - \log(1+p)$, as required.
\end{proof}

For the rest of this section, we set $x':= (x_1, \dots, x_{m-1})$.

\begin{nrmk} \label{multi-dom_rmk}
	Let $a \in \bar\CC^{m-1}$.  Set $\theta(a):= \theta - \max\{k' \cdot \|\im a\|:\ k \in K\}$ and $$\tau(a):=  (\{k_m:\ k \in K\},R_m,r,\theta(a), \Pi_m(\Delta)),$$  where $\Pi_m:\RR^m \into \RR$ is the projection on the last coordinate.  Note that if\ $|\im a|$ is sufficiently small, then $\theta(a) > \pi/2$.
	
	If $\theta(a) > 0$, then $S^{\tau(a)}$ is contained in the fiber $(S^\tau)_a$ of $S^\tau$ over $a$. Moreover, if $\re a < R'$ then, for each $p \in \NN$, the set $H^{k_m}_p\left(\log R_m\right)$ is contained in the fiber of the set $H^k_p(\log R)$ over $a$.  Therefore, if $\theta(a) > 0$ and $\re a < \log R'$, then $S^{\tau(a)}_p$ is contained in the fiber $\left(S^\tau_p\right)_a$, for each $p$.
	
	%
		%
\end{nrmk}

\begin{lemma} \label{fixing_lemma}
	Let $f \in \G_\tau$ and $a \in \bar\CC^{m-1}$ be such that $\theta(a) > 0$ and $\re a < \log R'$.  Then the function $f_a: S^{\tau(a)} \into \CC$, defined by $f_a(w):= f(a,w)$, belongs to $\G_{\tau(a)}$ and satisfies $\|f_a\|_{\tau(a)} \le \|f\|_\tau$ and $T(f_a)(X_m) = T(f)(e^a,X_m)$.
\end{lemma}

\begin{proof}
	Choose $f_p$, for $p \in \NN$, such that $f =_\tau \sum f_p$.  For $p \in \NN$, let $f_{a,p}:S^{\tau(a)}_p \into \CC$ be defined by $f_{a,p}(w):= f_p(a,w)$; these functions are well defined by \cref{multi-dom_rmk}, and $f_a(w) = \sum_p f_{a,p}(w)$ for every $w \in S^{\tau(a)}$.  
	
	For each $p$ we set $F_{a,p}(X_m):= T(f_p)(e^a,X_m)$, a generalized power series in the indeterminate $X_{m}$.  Since $\|T(f_p)\|_{R,k,p} < \infty$ and since, by \cref{multi-dom_rmk}, the polyradius  $\rho(b):= \left( |e^a|,b \right)$ belongs to $\cl\left(D^k_p\left(R\right)\right)$ for every $b \in \cl\left(D^{k_m}_p\left({R_m}\right)\right)$, we have that $$\|F_{a,p}\|_b = \|T(f_p)\|_{\rho(b)} \le \|T(f_p)\|_{R,k,p};$$ in particular, $\|F_{a,p}\|_{R_m,k_m,p} \le \|T(f_p)\|_{R,k,p}$ and $f_{a,p} = \bar{F_{a,p}}$, for each $p$.  
	Therefore $f_a =_{\tau(a)} \sum f_{a,p}$, that is, $f_a \in \G_{\tau(a)}$.  Since the inequality $\sum \|F_{a,p}\|_{R_m,k_m,p}\ r^p \le \sum \|T(f_p)\|_{R,k,p}\ r^p$ holds for all choices of $\{f_p\}$, we also get $\|f_a\|_{\tau(a)} \le \|f\|_\tau$.  \medskip
	
	\noindent\textbf{Claim.} The series $T(f)(e^a,X_m)$ belongs to $\Ps{R}{X_m^*}$ and is equal to $T(f_a)(X_m)$.\medskip
	
	To see this claim, since $\sum \|T(f_p)\|_{R,k,p}\ r^p < \infty$ and $\left(|e^a|,\rho^{\tau(a)}_p\right) \in \cl\left(D^k_p\left(R\right)\right)$ for each $p$, it follows from \cref{containment} that for each $\alpha_m \in [0,\infty)$ and each $p$,
	\begin{align*}
	\sum_{\alpha' \in [0,\infty)^{m-1}} |a_{(\alpha',\alpha_m),p}|\ |e^{\alpha'\cdot a}| &\le \frac{\|Tf_p\|_{R,k,p}}{\left(\rho_p^{\tau(a)}\right)^{\alpha_m}} \\
	&= \frac{\|Tf_p\|_{R,k,p}}{R_m^{\alpha_m}}\ (1+p)^{\alpha_m/k_m} \\
	&\le \frac{\|Tf_p\|_{R,k,p}}{R_m^{\alpha_m}}\ s^p\ (1+p)^{\alpha_m/k_m},
	\end{align*}
	for any $s > 1$.  In particular, for $s \in (1,r)$ we have, for each $\alpha_m \in [0,\infty)$, that
	\begin{equation} \label{b-finite}
	\sum_p \left(\sum_{\alpha'}|a_{(\alpha',\alpha_m),p}|\ |e^{\alpha'\cdot a}|\right) \le \frac1{R_m^{\alpha_m}} \sum_p \|Tf_p\|_{R,k,p}\ r^p \left(\frac sr\right)^p (1+p)^{\alpha_m/k_m} < \infty,
	\end{equation}
	which proves that $T(f)(e^a,X_m)$ belongs to $\Ps{R}{X_m^*}$.
	Moreover, since $f_a \in \G_{\tau(a)}$, we have $$T(f_a)(X_m) = \sum_p T(f_{a,p})(X_m) = \sum_p F_{a,p}(X_m).$$  
	
	It follows from \eqref{b-finite} that for all $\alpha_m$,
\begin{equation}
	\sum_{\alpha'}a_{(\alpha',\alpha_m)}e^{\alpha'\cdot a}=\sum_{\alpha'}\left(\sum_p a_{(\alpha',\alpha_m),p}\right)e^{\alpha'\cdot a}=\sum_p \sum_{\alpha'}a_{(\alpha',\alpha_m),p}e^{\alpha'\cdot a}.
	\end{equation}
Hence, $T(f)(e^a,X_m) = \sum_p F_{a,p}(X_m) = T(f_a)(X_m)$, as claimed.  
\end{proof}

\begin{prop}[Quasianalyticity] \label{multi_algebra_prop}
	The map $T:\G_\tau \into \Ps{C}{X^*}$ is an injective $\CC$-algebra homomorphism.
\end{prop}

\begin{proof}
	For the injectivity of $T$, let $f \in \G_\tau$ be such that $T(f) = 0$; we need to show that $f = 0$.  If $m=1$, this is done in  \cref{watson_cor}, so we assume $m > 1$.  Let $a \in \bar\CC^{m-1}$, and define $\theta(a)$ and $\tau(a)$ as in  \cref{multi-dom_rmk}.  Assume that $\theta(a) > \pi/2$ and $\re a < R'$; by \cref{fixing_lemma} and the assumption $Tf = 0$, we obtain $f_a \in \G_{\tau(a)}$ with $T(f_a) = 0$.  It follows from quasianalyticity of $\G_{\tau(a)}$ that $f_a = 0$.  Since the set of $a \in \bar\CC^{m-1}$ for which the latter holds contains an open set (by \cref{fixing_lemma}) and $f$ is holomorphic on its (connected) domain, it follows that $f=0$, as claimed.
\end{proof}

As in \cite[Cor. 2.19]{Dries:2000mx}, we now obtain

\begin{cor} \label{real_G}
	Let $f \in \G_\tau$.  Then $f(-\infty,R) \subseteq \RR$ if and only if $Tf \in \Ps{R}{X^*}$. \qed
\end{cor}


\subsection{Monomial division} \label{mon_div_section}
Let $F = \sum_{a \in [0,\infty)^m} a_\alpha X^\alpha \in \Ps{C}{X^*}$.  Recall from \cite[Section 4]{Dries:1998xr} that $$\ord(F)= \begin{cases}
\min\{|\alpha|:\ a_\alpha \ne 0\} &\text{if } f \ne 0, \\ \infty &\text{if } f = 0.
\end{cases}$$
For $i \in \{1, \dots, m\}$, we also consider $F$ as an element of $\Ps{C}{X_1^*, \dots, X_{i-1}^*, X_{i+1}^*, \dots, X_m^*}\left[\!\left[X_i^*\right]\!\right]$, and we denote by $\ord_i(F)$ the corresponding order of $F$ in the indeterminate $X_i$.  Note that $\ord_i(F)>0$ implies $\ord(F)>0$, for each $i$.

\begin{lemma} \label{mon_div_lemma}
	Let $f \in \G_\tau$ and assume that $\gamma:= \ord_m(Tf) > 0$.  Let also $s \in (1,r)$ and set $\tau':= (K,R,s,\theta,\Delta)$.  Then there exist $g \in \G_{\tau'}$ and $C = C(s/r) > 0$  depending only on $\frac sr$ such that $\|g\|_{\tau'} \le C \|f\|_{\tau'}$ and $$f(w) = e^{\gamma w_m} g(w) \quad\text{for } w \in S^{\tau'}.$$
\end{lemma}

\begin{proof}
	For simplicity, we write $\rho_p = \rho_p^\tau = \rho_p^{\tau'}$ and $S_p = S_p^\tau = S_p^{\tau'}$, for each $p$; recall that $\rho_p \in \cl\left(D^k_p(R)\right)$ for each $p$.  Say $f =_\tau \sum_p f_p$ with $\sum_p \|Tf_p\|_{R,k,p}\ r^p \le 2\|f\|_\tau$ and $\sum_p \|f_p\|_{S_p}\ r^p \le 2\|f\|_\tau$; since each $Tf_p$ is convergent we may assume, after replacing each $f_p$ by $f_p - \bar{(Tf_p)_\gamma}$ if necessary\footnote{where $(Tf_p)_\gamma$ denotes the truncation of $Tf_p$ at the exponent $\gamma$ with respect to the indeterminate $X_m$}, that $\ord_m(Tf_p) \ge \gamma$ for each $p$.  So there are $G_p \in \Ps{C}{X^*}$ with support contained in $\supp(Tf_p)$ (and hence natural) such that $Tf_p = X_m^\gamma G_p$; since $\|Tf_p\|_{\rho_p} = \rho_{p,m}^\gamma \|G_p\|_{\rho_p}$, it follows that $G_p$ has radius of convergence at least $\rho_p$, and that $$f_p(w) = e^{\gamma w_m} g_p(w) \quad\text{for } w \in H(\log\rho_p),$$ where $g_p:= \bar{G_p}$ for each $p$.  We extend $g_p$ to all of $S_p$ by setting $g_p(w):= f_p(w)/e^{\gamma w_m}$ for $w \in S_p\setminus H(\log\rho_p)$.
	
	Since $\|Tf_p\|_{\rho_p} = \rho_{p,m}^\gamma \|Tg_p\|_{\rho_p}$ for each $p$, we get
	\begin{equation*}
		\sum_p \|Tg_p\|_{\rho_p} s^p = \sum_p \frac{\|Tf_p\|_{\rho_p}}{\rho_{p,m}^\gamma} s^p \le \frac1{R_m^\gamma} \sum_p \|Tf_p\|_{\rho_p} r^p (1+p)^{k_m} \left(\frac sr\right)^p \le C \|g\|_\tau
	\end{equation*}
	for some $C = C(s/r)>0$ depending only on $\frac sr$.  It follows, on the one hand, that $$\sum_p \|g_p\|_{H(\log \rho_p)} s^p \le C \|g\|_\tau$$ as well.  On the other hand, since $\|Tg_p\|_{S_p \setminus H(\log\rho_p)} \le \rho_{p,m}^{-\gamma}\|Tf_p\|_{S_p\setminus H(\log\rho_p)}$, the same argument as above also proves that $\sum_p \|g_p\|_{S_p \setminus H(\log\rho_p)} s^p \le C \|g\|_\tau$.  Therefore, the function $g:S^{\tau'} \into \CC$ defined by $g:= \sum_p g_p$ belongs to $\G_{\tau'}$ and satisfies $\|g\|_{\tau'} \le C \|g\|_\tau$ and $f(w) = e^{\gamma w_m} g(w)$ for $w \in S^{\tau'}$, as claimed.
\end{proof}

\subsection{Generalized multisummable germs} \label{gen_mult_germs_section}
Similar to Section 2 of \cite{Dries:2000mx}, we let $(\T_m, \le)$ be the directed set of all tuples $\tau = (K,R,r,\theta, \Delta)$ as above, where $\tau = (K,R,r,\theta,\Delta) \le \tau' = (K',R',r',\theta',\Delta')$ if and only if $K \supseteq K'$, $R \le R'$, $r \le r'$, $\theta \le \theta'$ and $\Delta \supseteq \Delta'$.  Then $S^{\tau'} \subseteq S^\tau$ whenever $\tau' \le \tau$ and in this situation, for $f \in \G_\tau$, the restriction $f\rest{S^{\tau'}}$ belongs to $\G_{\tau'}$.  The directed limit of the directed system $\left(\G_\tau:\ \tau \in \T_m\right)$ under these restrictions is the set $\G_m$ of \textbf{germs at $-\infty$} of functions in $\G_\tau$, as $\tau$ ranges over $\T_m$.  This $\G_m$ is a $\CC$-algebra containing the germs at $-\infty$ of the functions $e^{\gamma z_i}$, for $i=1, \dots, m$ and $\gamma \ge 0$, and we extend each norm $\|\cdot\|_\tau$ to all of $\G_m$ by setting $\|f\|_\tau := \infty$ whenever $g \notin \G_\tau$.

Below, we write $-\infty:=(-\infty, \dots, -\infty)$.
Note that $f(-\infty) = (Tf)(0)$, since $Tf$ is the asymptotic expansion of $f$ at $-\infty$; hence $f(-\infty) = 0$ if and only if $\ord(Tf)>0$.

\begin{lemma} \label{directed_lemma}
	Let $f \in \G_{m}$.
	\begin{enumerate}
		\item For $g \in \G_m$ and $\tau \in \T_m$, we have $\|fg\|_\tau \le \|f\|_\tau \|g\|_\tau$.
		\item If $f(-\infty) = 0$, then $\,\lim_{\tau} \norm{f}_{\tau} = 0$, where the limit is taken over the downward directed set $\T_m$.
		\item If $f(-\infty) \neq 0$, then $f$ is a unit in $\G_{m}$.
		\item If $f(-\infty) = 0$, then there are $\gamma_i > 0$ and $f_i \in \G_{m}$ for $i =1,
		\mdots,  m$ such that $f = e^{\gamma_1 z_1} \cdot f_1 + \cdots + e^{\gamma_m z_m} \cdot f_m$.
		\item If $m > 1$, then the germ $f(-\infty,\cdot)$ belongs to $\G_{m-1}$.
	\end{enumerate}
\end{lemma}

\begin{proof}
	Parts (2)--(4) are similar to Lemma 2.14 in \cite{Dries:2000mx}, using \cref{mon_div_lemma} for part (4).  Part (5) follows from \cref{fixing_lemma} with $\nu = 1$.
\end{proof}

Similar to \cite[Section 3]{Dries:2000mx}, we set $\Pc{C}{X^*}_\tau := T(\G_\tau)$, for $\tau \in \T_m$, and $\Pc{C}{X^*}_\G := T(\G_m)$. We refer to the latter as the $\CC$-algebra of all \textbf{generalized multisummable series} in the indeterminates $X$; note that $\Pc{C}{X^*}_\G = \bigcup_{\tau \in \T_m} \Pc{C}{X^*}_\tau$.  We make each $\Pc{C}{X^*}_\tau$ into an isomorphic copy of $\G_\tau$ by setting $\|Tf\|_\tau := \|f\|_\tau$, for $f \in \G_\tau$; and we extend each norm $\|\cdot\|_\tau$ to all of $\Pc{C}{X^*}_\G$ by setting $\|F\|_\tau:= \infty$ for $F \notin \Pc{C}{X^*}_\tau$.

Extending our notation of log-sum of convergent generalized power series, we also shall write $\bar F:= T^{-1}(F)$, for $F \in \Pc{C}{X^*}_\tau$.

\subsection{Mixed series} \label{mixed_section}
We now consider additional indeterminates $Y = (Y_1, \dots, Y_n)$, and we defined mixed series similar to \cite[Section 3]{Dries:2000mx}.  Thus, for $\tau \in \T_m$ and $\rho \in (0,\infty)^n$, and for $F = \sum_{\beta \in \NN^n} F_\beta(X) Y^\beta \in \Pc{C}{X^*}_\tau [\![Y]\!]$, we set
\begin{equation*}
	\|F\|_{\tau,\rho}:= \sum_\beta \|F_\beta\|_\tau \rho^\beta
\end{equation*}
and
\begin{equation*}
	\Pc{C}{X^*;Y}_{\tau,\rho} := \Pc{C}{X^*}_\tau \left\{Y\right\}_\rho = \set{F \in \Pc{C}{X^*}_\tau [\![Y]\!]:\ \|F\|_{\tau,\rho} < \infty}.
\end{equation*}
The latter is a Banach $\CC$-algebra with respect to the former norm.
We sometimes refer to the indeterminates $X_i$ as the \textbf{generalized Gevrey} variables and to the indeterminates $Y_j$ as the \textbf{convergent} variables.  Each $F = \sum_\beta F_\beta(X)Y^\beta \in \Pc{C}{X^*;Y}_{\tau,\rho}$ defines a log-holomorphic function $\bar F:S^\tau \times D(\rho) \into \CC$ by setting $$\bar F(w,y):= \sum_\beta \bar{F_\beta}(w)y^\beta.$$  As in \cref{gen_mult_germs_section}, we set $$\Pc{C}{X^*;Y}_\G := \bigcup_{\tau \in \T_m, \rho \in (0,\infty)^n} \Pc{C}{X^*;Y}_{\tau,\rho},$$  and we extend each norm $\|\cdot\|_{\tau,\rho}$ to all of $\Pc{C}{X^*;Y}_\G$ by setting $\|F\|_{\tau,\rho}:= \infty$ whenever $F \notin \Pc{C}{X^*;Y}_{\tau,\rho}$.  Ordering the product $\T_m \times (0,\infty)^n$ by the product order, we obtain the following generalization of \cref{directed_lemma}:

\begin{lemma} \label{mixed_directed_lemma}
	Let $F \in \Ps{C}{X^*;Y}_\G$.
	\begin{enumerate}
		\item For $G \in \Pc{C}{X^*;Y}_\G$, $\tau \in \T_m$ and $\rho \in (0,\infty)^n$, we have $\|FG\|_{\tau,\rho} \le \|F\|_{\tau,\rho} \|G\|_{\tau,\rho}$.
		\item If $\bar F(-\infty,0) = 0$, then $\,\lim_{(\tau,\rho)} \norm{F}_{\tau,\rho} = 0$, where the limit is taken over the downward directed set $\T_m \times (0,\infty)^n$.
		\item If $\bar F(-\infty,0) \neq 0$, then $F$ is a unit in $\Pc{C}{X^*;Y}_\G$. 
		\item If $m > 1$ and $X' = (X_2, \dots, X_m)$, then $F(0,X',Y)$ belongs to $\Pc{C}{(X')^*;Y}_\G$.
		\item $\Pc{C}{X^*;Y}_\G$ is complete in each norm $\|\cdot\|_{\tau,\rho}$.
		\item $\Pc{C}{X^*;Y}_\G \subseteq \Pc{C}{(X,Y)^*}_\G$.
	\end{enumerate}
\end{lemma}

\begin{proof}
	Parts (1)--(4) follow from \cref{directed_lemma}.
	Part (5) is just a restatement of the fact that each $\Pc{C}{X^*;Y}_{\tau,\rho}$ is a Banach algebra.  Part (6) is proved along the lines of \cite[Lemma 3.5]{Dries:2000mx}.
\end{proof}

Recall that $F \in \Ps{C}{X^*;Y}$ is \textbf{regular in $Y_n$ of order $d$} if $F(0,0,Y_n) = uY_n +$ terms in $Y_n$ of higher degree, with $u \in \CC$ nonzero.  Using \cref{mixed_directed_lemma}, the proof of \cite[Proposition 4.1]{Tougeron:1994fk} now establishes the following, where $Y' = (Y_1, \dots, Y_{n-1})$:

\begin{prop}\label{weierstrass}
	Let $f \in \Pc{C}{X^*;Y}_\G$, and assume that $n > 0$ and $F$ is regular
	in $Y_n$ of order $d$. Then the series F factors uniquely as $F = GH$, where
	$G \in \Pc{C}{X^*;Y}_\G$ is a unit and $H \in \Pc{C}{X^*;Y'}_\G [Y_n]$ is monic in $Y_n$ of degree $d$. \qed
\end{prop}

\section{Substitutions}  \label{substitution}

In this section, we introduce the substitutions discussed in Sections 1.8 and 1.15 of \cite{MR3349791}.
Let $m',n' \in \NN$ and $X' = (X'_1, \dots, X'_{m'})$ and $Y' = (Y'_1, \dots, Y'_{n'})$ be indeterminates. Denote by $\{X,Y\}$ the set $\left\{ X_{1},\ldots,X_{m},Y_{1},\ldots,Y_{n}\right\}$ and let $\sigma :
\{X,Y\} \into \Ps{R}{(X')^*,Y'}$ be a map.  We call $\sigma$ a \textbf{substitution} if each $\sigma(X_i)$ is normal in the following sense:
\begin{itemize}
	\item[($\ast$)] there exist $a_i \in [0,\infty)$, nonzero $\gamma_i \in [0,\infty)^{m'}$, $\lambda_i \in (0,\infty)$ and $H_i \in \Ps{R}{(X')^*,Y'}$ such that $H_i(0) = 0$ and $\sigma(X_i) = a_i +  (X')^{\gamma_i}(\lambda_{i} + H_i((X')^*,Y'))$.  
\end{itemize}
If $\sigma$ is a substitution such that $\sigma(0) = 0$ (in particular, $a_i = 0$ for each $i$), then $\sigma$ extends to a unique $\CC$-algebra homomorphism $\sigma:\Ps{C}{X^*, Y}
\into \Ps{C}{(X')^*,Y'}$ by using, for each $r>0$ and $\epsilon \in \Ps{C}{(X')^*,Y'}$ with $\epsilon(0) = 0$,  the binomial theorem $$(\lambda + \epsilon)^r = \sum_{i \in \NN} \binom{r}{i} \lambda^{r-i} \epsilon^i.$$
In this situation, we write $\sigma F$ in place of $\sigma(F)$ for $F \in
\Ps{C}{X^*, Y}$.
If all $\sigma(X_i)$ and $\sigma(Y_j)$ lie in a subring $A$ of
$\Ps{C}{(X')^*,Y'}$, then we refer to $\sigma$ also as a substitution $\sigma
: \{X,Y\} \into A$.   Note that, in this situation, we have $\sigma(\Ps{R}{X^*,Y}) \subseteq \Ps{R}{(X')^*,Y'}$.

\begin{rmk}
	While general substitutions do not extend to all of $\Ps{C}{X^*,Y}$, we describe particular substitutions below that do extend to certain subalgebras of $\Pc{C}{X^*,Y}_\G$.
\end{rmk}

Let $\sigma:\{X,Y\} \into \Pc{R}{(X')^*,Y'}_{\G}$ be a substitution.  For each $\sigma(X_i)$, we let $\gamma_i$, $\lambda_i$ and $H_i$ be as in ($\ast$); by \cref{mixed_directed_lemma}, we have $H_i \in \Pc{R}{(X')^*,Y'}_{\G}$ as well.  We call $\tau' \in \T_\mu$ and $\rho' \in (0,\infty)^\nu$ \textbf{$\sigma$-admissible} if each $\sigma(X_i)$ and $\sigma(Y_j)$, as well as each $H_i$, belongs to $\Pc{R}{(X')^*,Y'}_{\tau',\rho'}$, and we have $\|H_i\|_{\tau',\rho'} < |\lambda_i|$ for each $i$.

Let $\tau' \in \T_\mu$ and $\rho' \in (0,\infty)^\nu$ be \textbf{$\sigma$-admissible}.  Then $\sigma$ induces a log-holomorphic map $\tilde\sigma: S^{\tau'} \times D(\rho') \into \bar\CC^{m+n}$ by setting
\begin{equation*}
\tilde\sigma_i(w',y'):= \gamma_i \cdot w' + \log\left(\lambda_i + \bar{H_i}(w',y')\right)
\end{equation*}
for $i = 1, \dots, m$ (where $\log$ denotes the standard branch of the logarithm), and 
\begin{equation*}
\tilde\sigma_j:= \bar{\sigma(Y_j)}
\end{equation*}
for $j=1, \dots, n$.  (Note that, for each $i$, we have $\exp \circ\ \tilde\sigma_i = \bar{\sigma(X_i)}$.)

\begin{expls} \label{subst_expls}
	\begin{enumerate}
		\item \textbf{(Permutation of Gevrey variables)}  For a permutation $\pi \in \Sigma_m$, $\mu = m$ and $\nu = n$, the substitution defined by $\sigma(X_i):= X'_{\pi(i)}$ and $\sigma(Y_j):= Y_j$.
		\item \textbf{(Blow-up chart in the Gevrey variables)}  $\sigma$ is any of the blow-up charts (1) or (2) found in \cite[Definition 1.13]{MR3349791}, also referred to as \textbf{regular} blow-up charts and \textbf{singular} blow-up charts, respectively.
		\item \textbf{(Ramification of a Gevrey variable)}  Here $m'=m$ and $n'=n$, and we have $\sigma(X_{i_0}) = (X'_{i_0})^\alpha$ for some $\alpha \ge 0$ and $i_0 \in \{1, \dots, m\}$, $\sigma(X_i) = X'_i$ for each $i \ne i_0$, and $\sigma(Y_j) = Y'_j$ for each $j$.
		\item \textbf{(Translation)}  For $(a,b)\in (0,\infty)^m \times \RR^n$, put $m':=|\{i:\ 1\le i\le m,\ a_i =0\}|$ and $n':= n+m-m'$, and choose a permutation $\pi \in \Sigma_m$ such that $\pi(\{i:\ 1 \le i \le m,\ a_i = 0\}) = \{1,...,m'\}$.  Then the \textbf{translation by $(a,b)$} is the substitution defined by $\sigma(X_{\pi(i)}):= X'_i$ for $i = 1, \dots, m'$, $\sigma(X_{\pi(m'+j)}):= a_{\pi(m'+j)} + Y'_j$ for $j = 1, \dots, m-m'$, and $\sigma(Y_j):= b_j + Y'_{m-m'+j}$ for $j=1, \dots, n$.  
		\item \textbf{(Infinitesimal substitution in the convergent variables)} Here $n>0$, $\sigma(0) = 0$, $\sigma(X_i) = X_i$ for each $i$ and $\sigma(Y_j) \in \Pc{R}{(X')^*,Y'}_\G$ for each $j$, where $X' = (X'_1, \dots, X'_{m'})$ and $Y' = (Y'_1, \dots, Y'_{n'})$ with $m'$ and $n'$ arbitrary.  
	\end{enumerate}\medskip

	Note that for each of these substitutions $\sigma$, every corresponding $\tau'$ and $\rho'$ is $\sigma$-admissible.  Also, while permutations and blow-up charts extend to all of $\Ps{C}{X^*,Y}$, translations do not in general do so.
\end{expls}

We start with an elementary lemma similar to \cite[Lemma 4.3]{Dries:2000mx}.  The essential difference between the proofs here and those in \cite[Section 4]{Dries:2000mx} is that we cannot use Taylor expansion to compute the series after substitution as in \cite[Lemma 4.2]{Dries:2000mx}; instead, we have to rely on our additional assumptions built into the norms $\|\cdot\|_\tau$.

\begin{lemma} \label{one_var_for_another}
	Let $X':= (X_1, \dots, X_{m-1})$, and let $\sigma:\{X,Y\} \into \RR[X',Y]$ be the substitution given by $\sigma(X_i) := X_i$ if $i < m$, $\sigma(X_m):= X_{m-1}$ and $\sigma(Y_j):= Y_j$ for each $j$.  Then for every $F \in \Pc{C}{X^*,Y}_\G$, we have $\sigma F \in \Pc{C}{(X')^*, Y}_\G$ and $\bar{\sigma F} = \bar F \circ \tilde\sigma$ (as germs of functions).
\end{lemma}

\begin{proof}
	Fix $\tau = (K,R,r,\theta,\Delta) \in \T_m$ and $\rho \in (0,\infty)^n$.
	Similar to \cite[Lemma 4.3]{Dries:2000mx}, we set $$K':= \set{(k_1, \dots, k_{m-2}, k_{m-1} + k_m):\ k \in K} \quad\text{and}\quad R':= (R_1, \dots, R_{m-2}, \min\{R_{m-1}, R_m\}).$$  Moreover, we set $$\Delta':= \set{(\alpha_1, \dots, \alpha_{m-2}, \alpha_{m-1} + \alpha_m):\ \alpha \in \Delta};$$ since the projection of $\Delta'$ on each of the first $m-2$ coordinates is the same as of $\Delta$, and since $\Pi_{m-1}(\Delta') \subseteq \Pi_{m-1}(\Delta) + \Pi_m(\Delta)$, this $\Delta'$ is natural.  So we set $\tau':= (K',R',r,\theta,\Delta') \in \T_{m-1}$; we get from the proof of \cite[Lemma 4.3]{Dries:2000mx} that $\tilde\sigma\left(S^{\tau'} \times D(\rho')\right) \subseteq S^\tau \times D(\rho)$ and, for each $p$, that $\tilde\sigma\left(S^{\tau'}_p \times D(\rho')\right) \subseteq S^\tau_p \times D(\rho)$.  The lemma then follows from the following more precise statement: \medskip
	
	\noindent\textbf{Claim.} \textsl{For $F \in \Pc{C}{X^*,Y}_{\tau,\rho}$, we have $\sigma F \in \Pc{C}{(X')^*,Y}_{\tau',\rho}$ such that $\|\sigma F\|_{\tau'} \le \|F\|_\tau$ and $\bar{\sigma F} = \bar F \circ \tilde\sigma$}.  \medskip
	
	To prove the claim, let $F \in \Pc{C}{X^*,Y}_{\tau,\rho}$; we distinguish two cases.\medskip
	
	\textbf{Case 1:} $n=0$.  Choose convergent $F_p \in \Pc{C}{X^*,Y}$ such that $\bar F =_\tau \sum_p \bar{F_p}$; as in Case 1 of the proof of \cite[Lemma 4.3]{Dries:2000mx}, it follows that $\sum_p \left\|\bar{F_p}\right\|_{S^{\tau'}_p} r^p \le \|F\|_\tau$ and $\bar{\sigma F} = \bar F \circ \tilde \sigma$.  Also, let $k'$ be defined for $K'$ as $k$ is defined for $K$; it remains to show that $\sum_p \|\sigma F_p \|_{R',k',p}\ r^p < \infty$.  Let $s \in D^{k'}_p\left(e^{R'}\right)$; then we have $\|\sigma F_p\|_s = \|F_p\|_{\sigma(s)}$, for each $p \in \NN$, by the definition of these norms.  Since $\sigma(s) \in D^k_p\left(e^R\right)$ by the above, we obtain $$\sum_p \|\sigma F_p\|_s\ r^p = \sum_p \|F_p\|_{\sigma(s)}\ r^p \le \sum_p \|F_p\|_{R,k,p}\ r^p.$$ Since $s \in D^{k'}_p\left(e^{R'}\right)$ was arbitrary, it follows that $\sum_p \|\sigma F_p \|_{R',k',p}\ r^p  \le \sum_p \|F_p\|_{R,k,p}\ r^p$, so that $\sigma F \in \Pc{C}{(X')^*,Y}_{\tau',\rho}$.   Moreover, since this argument works for all sequences of convergent $F_p \in \Pc{C}{X^*,Y}$ such that $\bar F =_\tau \sum_p \bar{F_p}$, we also get $\|\sigma F\|_{\tau'} \le \|F\|_\tau$ in this case. \medskip
	
	\textbf{Case 2:} $n>0$; this case literally follows the proof of Case 2 of \cite[Lemma 4.3]{Dries:2000mx}, which we reproduce here for the convenience of the reader.  We let $F = \sum F_{\beta}(X) Y^{\beta}$ with each $F_{\beta} \in
	\Pc{C}{X^*}_{\tau}$.
	By Case 1 each $\sigma F_{\beta}$ belongs to
	$\Pc{C}{(X')^*}_{\tau'}$ with $\norm{\sigma F_{\beta}}_{\tau'}
	\leq \norm{F_{\beta}}_{\tau}$.
	Hence $\sigma F$ belongs to $\Pc{C}{(X')^*;Y}_{\tau',\rho}$ and
	satisfies $\norm{\sigma F}_{\tau',\rho} \leq
	\norm{F}_{\tau,\rho}$, and it remains to show that $\bar{\sigma
	F} = \bar F \circ
	\tilde\sigma$.
	For each $d \in \NN$ we put
	$$F_d(X, Y):= \sum_{\Sigma\beta \leq d} F_{\beta}(X) Y^{\beta} \in
	\Pc{C}{X^*;Y}_{\tau,\rho}.$$
	By the same argument as before each $\sigma F_d$ belongs to
	$\Pc{C}{(X')^*, Y}_{\tau',\rho}$, and since each $F_d$ has finite
	support (as a series in $Y$), we get $\bar{\sigma F_d} =
	\bar{F_d} \circ \tilde\sigma$.
	Clearly $\lim_{d \to \infty} \bar{F_d}(w,y) = \bar{F}(w,y)$ for all $(w,y) \in S^{\tau} \times D(\rho)$.
	Moreover, since $\sigma:\Ps{C}{X;Y} \into \Ps{C}{X', Y}$ is a
	homomorphism, we have
	$$\norm{\bar{\sigma F} - \bar{\sigma F_d}}_{\tau',\rho} = \norm{\bar{\sigma(F - F_d)}
	}_{\tau',\rho} = \sum_{\Sigma\beta > d}
	\norm{\bar{\sigma F_{\beta}}}_{\tau'} \rho^{\beta}
	\leq \sum_{\Sigma\beta > d}
	\norm{F_{\beta}}_{\tau} \rho^{\beta},$$
	so that $\lim_{d \to \infty} \bar{\sigma F_d}(w',y) = \bar{\sigma F}(w',y)$ for
	all $(w',y) \in S^{\tau'} \times D(\rho)$.
	Hence
	$$\bar{\sigma F}(w',y)
	= \lim_{d \to \infty} \bar{\sigma F_d} (w',y)
	= \lim_{d \to \infty} \big(\bar{F_d} \circ
	\tilde\sigma\big) (w',y)
	= \big(\bar{F} \circ \tilde\sigma\big)(w',y)$$
	for all $(w',y) \in S^{\tau'} \times D(\rho)$, which finishes the proof.
\end{proof}

We now proceed to proving closure under each of the substitutions in \cref{subst_expls}.

\subsection*{Permutations}
Let $\sigma$ be a permutation of $\{1, \dots, m\}$ and, for $x \in \bar\CC^m$, we denote by $\sigma(x)$ the tuple $ \left(x_{\sigma(1)}, \dots, x_{\sigma(m)}\right)$.
For $\tau \in \T_m$, we set $$\sigma(\tau):= \left(K,\sigma(R),r,\theta,\sigma(\Delta)\right);$$ note that $\sigma:\T_m \into \T_m$ is a bijection.

\begin{lemma} \label{permutation_lemma}
	For $\tau \in \T_m$ and $\rho \in (0,\infty)^n$, we have $F \in \Pc{C}{X^*,Y}_{\tau,\rho}$ if and only if $\sigma F \in \Pc{C}{X^*,Y}_{\sigma^{-1}(\tau),\rho}$, and for such $F$, we have $\bar{\sigma F} = \bar F \circ \tilde\sigma$ and $\|\sigma F\|_{\sigma^{-1}(\tau),\rho} = \|F\|_{\tau,\rho}$.  
\end{lemma}

\begin{proof}
	The proof follows the general strategy of the proof \cref{one_var_for_another}, but is easier and left to the reader.
\end{proof}

\subsection*{Blow-up charts}
	We let $1 \le j < i \le m$ and $\lambda > 0$ and consider the regular blow-up chart $\pi^\lambda_{i,j}$; the singular blow-up charts are handled similarly, but are actually easier to deal with and are left to the reader.  Permuting the Gevrey variables if necessary, we assume that $j = m-1$ and $i = m$; to simplify notation, we write $\sigma = \pi^\lambda_{m,m-1}$. 
	
	Fix $\tau \in \T_m$ and $\rho \in (0,\infty)^n$.  Similar to \cite[Lemma 4.7]{Dries:2000mx}, we now choose $\theta' \in (\pi/2,\theta)$ and $\rho_0 > 0$ such that $k_m |\arg(\lambda + v)| < \theta - \theta'$ for all $v \in D(2\rho_0)$ and all $k \in K$.  We set $k':= (k_1, \dots, k_{m-2}, k_{m-1}+k_m)$ for $k \in K$, $l:= \max\{k_m:\ k \in K\}$ and $R_{m-1}':= \min\set{R_{m-1}, R_m, \frac{R_m}{(\lambda + 2\rho_0)^l}}$, as well as $$K':= \set{k':\ k \in K} \quad\text{and}\quad R':= (R_1, \dots, R_{m-2}, R_{m-1}'),$$ $\Delta':= \Pi_{m-1}(\Delta)$ and, finally, $\tau':= (K',R',r,\theta',\Delta')$ and $\rho':= (\rho_0,\rho)$.  By Claims 1 and 2 in the proof of \cite[Lemma 4.7]{Dries:2000mx}, we have $\tilde\sigma\left(S^{\tau'} \times D(\rho')\right) \subseteq S^\tau \times D(\rho)$ and, for each $p$, that $\tilde\sigma\left(S^{\tau'}_p \times D(\rho')\right) \subseteq S^\tau_p \times D(\rho)$. 
	
\begin{prop}\label{blow-up_prop}
	For $F \in \Pc{C}{X^*,Y}_{\tau,\rho}$, we have $\sigma F \in \Pc{C}{(X')^*,Y'}_{\tau',\rho'}$ such that $\bar{\sigma F} = \bar{F} \circ \tilde\sigma$ and $\|\sigma F\|_{\tau',\rho'} \le C \|F\|_{\tau,\rho}$, where $C \ge 1$ is independent of $F$.
\end{prop}
	
\begin{proof}
	Let $F \in \Pc{C}{X^*,Y}_{\tau,\rho}$; again we distinguish two cases. \medskip
	
	\textbf{Case 1:} $n=0$.  Then $F \in \Pc{C}{X^*}_\tau$, and we choose convergent $F_p \in \Pc{C}{X^*}$ such that $\bar F =_\tau \sum_p \bar{F_p}$.  As in the proof of \cite[Lemma 4.7]{Dries:2000mx}, for each $p, \nu \in \NN$ we define $f_{p,\nu}:S^{\tau'}_p \into \CC$ and $f_{\nu}:S^{\tau'} \into \CC$ by $$f_{p,\nu} (w') := \frac 1{\nu!} \frac {\partial^{\nu}
	\left(\bar{F_p} \circ \tilde\sigma\right)} {\partial v^{\nu}} (w', 0) \quad\text{and}\quad f_{\nu} (w') := \frac 1{\nu!} \frac {\partial^{\nu} \left(\bar F \circ \tilde\sigma\right)} {\partial v^{\nu}} (w', 0).$$  The argument there shows that, for each $\nu$, we have $f_\nu = \sum_p f_{p,\nu}$ on $S^{\tau'}$ with $\|f_{\nu}\|_{S^{\tau'}} \le \left\|\bar F\right\|_{S^\tau}/(2\rho_0)^\nu$.  Hence $\sum_\nu \|f_\nu\|_{S^{\tau'}} \rho_0^\nu \le \left\|\bar F\right\|_{S^\tau}$, and it follows from Taylor's Theorem that $$\left(\bar F \circ \tilde\sigma\right)(w',y) = \sum_\nu f_\nu(w') y^\nu \quad\text{ on } S^{\tau'} \times D(\rho_0).$$  So it remains to show that each $f_\nu$ belongs to $\G_{\tau'}$; to do so, we define $k'$ for $K'$ as $k$ was defined for $K$, and we establish the following \medskip

	\noindent\textbf{Claim.} \textsl{Each $f_{p,\nu}$ is given by a generalized power series $F_{p,\nu}$ with support in $\Delta'$ such that $\|F_{p,\nu}\|_{R',k',p} \le C \cdot \|F_p\|_{R,k,p}/(2\rho_0)^\nu$, where $C>0$ is independent of $F$, $p$ or $\nu$. } \medskip
	
	To see the claim, we use \cite[Lemma 6.5]{Dries:1998xr}---or more precisely, the following modification of it:  the stated hypotheses there, namely, that $\tau \le \rho$, $\tau_m < \lambda$ and $\tau_{m-1}^\gamma(\lambda + \tau_m) < \rho_m$, were sufficient for the purposes of that paper, but not quite necessary to obtain the same conclusion from the proof of that lemma.  Indeed, it suffices to assume that $\tau_i \le \rho_i$ for all $i \ne m$, that $\tau_m < \lambda$ and that $\tau_{m-1}^\gamma(\lambda + \tau_m) \le \rho_m$ to obtain the same conclusion, and we shall verify  these weaker hypotheses below (with $\gamma = 1)$ in order to apply that lemma here, without further mention of this discrepancy.
	
	On the one hand, by Taylor's Theorem, we have for each $p$ that $$\left(\bar{F_p} \circ \tilde\sigma\right)(w',y) = \sum_{\nu} f_{p,\nu}(w') y^\nu \quad\text{on } S^{\tau'}_p \times D(\rho_0).$$  On the other hand, using the binomial formula, we have for each $p$ that $$\sigma F_p = \sum_{\nu} F_{p,\nu}(X') \cdot X_m^\nu \quad\text{with}\quad F_{p,\nu}(X') = \frac1{\nu!} \frac{\partial^\nu (\sigma F_p)}{\partial X_m^\nu}(X',0);$$  note that each $F_{p,\nu}$ has support contained in $\Delta'$.
	
	We now fix an arbitrary $s' \in \cl\left(D^{k'}_p(\log R')\right) \cap (0,\infty)^{m-1}$ and set $s:= (s',2\rho_0)$ and $t:= \sigma(s)$.  By Claim 2 of \cite[Lemma 4.7]{Dries:2000mx}, we have $t \in \cl\left(D^k_p(\log R)\right)$.  Since $s_i = t_i$ for $i=1, \dots, m-1$, $s_m = 2\rho_0 < \lambda$ and $s_{m-1}(\lambda + s_m) = t_m$, we get from \cite[Lemma 6.5]{Dries:1998xr} a constant $C\ge 1$, independent of $F$, $p$ or $\nu$, such that $$\|\sigma F_p\|_s \le C \|F_p\|_t \le C \|F_p\|_{R,k,p}.$$  Since $\|\sigma F_p\|_s = \sum_\nu \|F_{p,\nu}\|_{s'} s_m^\nu$ by definition of $\|\cdot\|_s$, it follows that $$\|F_{p,\nu}\|_{s'} \le \frac C{(2\rho_0)^\nu} \|F_p\|_{R,k,p}, \quad\text{for each } p \text{ and } \nu.$$  Since $s' \in \cl\left(D^{k'}_p(\log R')\right) \cap (0,\infty)^{m-1}$ was arbitrary, we finally get $$\|F_{p,\nu}\|_{R',k',p} \le \frac C{(2\rho_0)^\nu} \|F_p\|_{R,k,p}, \quad\text{for each } p \text{ and } \nu.$$
	
	Finally, we get from \cite[Lemmas 5.9(2,3) and 6.3(4)]{Dries:1998xr} that $f_{p,\nu} = \bar{F_{p,\nu}}$, for each $p$ and $\nu$.  This finishes the proof of the claim. \medskip
	
	It follows from the claim that $f_\nu =_{\tau'} \sum_p f_{p,\nu}$.  Moreover, since the claim holds for all sequences $F_p$ such that $\bar F =_{\tau} \sum_p \bar{F_p}$, we also get that $\|f_\nu\|_{\tau'} \le C \cdot \|F\|_\tau/(2\rho_0)^\nu$ for each $\nu$.  It follows that $\|\sigma F\|_{\tau',\rho_0} = \sum_\nu \|f_\nu\|_{\tau'} \rho_0^\nu \le C \|F\|_\tau$, which finishes the proof of the proposition in Case 1. \medskip
	
	\textbf{Case 2:} $n>0$.  Let $F = \sum_{\beta \in \NN^n} F_\beta(X) Y^\beta$.  Then $\|\sigma F\|_{\tau',\rho'} = \sum_{\beta \in \NN^n}  \|\sigma F_\beta\|_{\tau',\rho_0}\ \rho^\beta$, so the proposition in Case 2 follows from Case 1 as in the proof of \cref{one_var_for_another}.
\end{proof}

\subsection*{Ramifications}
	Let $\sigma$ be a ramification of the Gevrey variable $X_{i_0}$ as in \cref{subst_expls}(3).  Permuting coordinates, we may assume that $i_0 = 1$.  Fix $\tau \in \T_m$ and $\rho \in (0,\infty)^n$.  We define $$K':= \set{(k_1/\alpha, k_2, \dots, k_m):\ k \in K},$$ $$R':= \left(R_1^{1/\alpha}, R_2, \dots, R_m\right)$$ and $$\Delta':= \set{(\beta_1/\alpha, \beta_2, \dots, \beta_m):\ \beta \in \Delta},$$ and we set $\tau':= (K',R',r,\theta, \Delta')$.  then $\Delta'$ is natural, and we have $\tilde\sigma\left(S^{\tau'} \times D(\rho)\right) \subseteq S^\tau \times D(\rho)$ and $\tilde\sigma\left(S^{\tau'}_p \times D(\rho)\right) \subseteq S^\tau_p \times D(\rho)$ for each $p$.  
	
\begin{prop} \label{ramification_prop}
	For $F \in \Pc{C}{X^*,Y}_{\tau,\rho}$, we have $\sigma F \in \Pc{C}{(X')^*,Y'}_{\tau',\rho}$ such that $\bar{\sigma F} = \bar{F} \circ \tilde\sigma$ and $\|\sigma F\|_{\tau',\rho} \le \|F\|_{\tau,\rho}$.
\end{prop}
	
\begin{proof} 
	We let $F \in \Pc{C}{X^*,Y}_{\tau,\rho}$, and we reduce to the case $n=0$ as in the proof of \cref{one_var_for_another}.  In this case, we let $k'$ be defined for $K'$ as $k$ is defined for $K$, and we choose convergent $F_p \in \Pc{C}{X^*}$ such that $\bar F =_{\tau} \sum_p \bar{F_p}$.  
	
	Fix $p$ and let $s \in D^{k'}_p(R') \cap (0,\infty)^m$ be a polyradius.  Then $\sigma(s) \in D^k_p(R)$ by the above and $\|\sigma F_p\|_s = \|F_p\|_{\sigma(s)} \le \|F_p\|_{R,k,p}$.  Since $s \in D^{k'}_p(R')$ was arbitrary, this shows that $\|\sigma F_p\|_{R',k',p} \le \|F_p\|_{R,k,p}$.  Since we obviously have $\|\bar{\sigma F_p}\|_{S^{\tau'}_p} \le \|\bar{F_p}\|_{S^\tau_p}$, it follows that $\sigma F \in \Pc{C}{(X')^*}_{\tau'}$ and $\bar{\sigma F} = \bar F \circ \tilde\sigma$.
	
	Moreover, since the above inequalities hold for all choices of $F_p$ such that $\bar F =_\tau \sum_p F_p$, it follows that $\|\sigma F\|_{\tau'} \le \|F\|_\tau$.
\end{proof}

\subsection*{Translations}  
Let $\sigma:\{X,Y\} \into \Pc{R}{(X,X')^*,Y'}_{\G}$ be the translation by a point $(a,b) \in [0,\infty)^m \times \RR^n$.

\begin{prop} \label{translation_prop}
	Let $\tau \in \T_m$ and $\rho \in (0,\infty)^n$ be  such that $\|(a,b)\| \le (R,\rho)$.  Then for every $F \in \Pc{C}{X^*,Y}_{\tau,\rho}$, we have $\sigma F \in \Pc{C}{(X')^*, Y'}_\G$ and $\bar{\sigma F} = \bar F \circ \tilde\sigma$.
\end{prop}

\begin{proof}
	We may assume that all but one of the coordinates of $(a,b)$ are zero.  If $a_i \ne 0$ for some $i$, we may assume, after permuting the Gevrey variables if necessary, that $i = m$.  This case is handled similarly to the proof of \cref{blow-up_prop}, except that we use \cite[Lemma 6.6]{Dries:1998xr} instead of \cite[Lemma 6.5]{Dries:1998xr} (with a similar weakening of hypotheses for the former as used above for the latter).  So we assume that $b_j \ne 0$ for some $j$; in this case, we adapt the usual Taylor expansion argument for convergent series to obtain the conclusion (details are left to the reader).
\end{proof}

\subsection*{Infinitesimal substitutions} 
	We recall the following observations:

\begin{lemma} \label{adding_vars}
	Let $\tau = (K,R,r,\theta,\Delta) \in \T_m$ and $X' = (X'_1, \dots, X'_n)$, and let $F \in \Pc{C}{X^*}_\tau$.
	\begin{enumerate}
		\item For any $\tau' \geq \tau$, we have $F \in \Pc{C}{X^*}_{\tau'}$ with $\|F\|_{\tau'} \le \|F\|_\tau$.
		\item The series $G(X,X'):= F(X)$ belongs to $\Pc{C}{(X,X')^*}_{\tau'}$ with $\|G\|_{\tau'} = \|F\|_{\tau}$, where $\tau' = (K',R',r,\theta,\Delta')$ with $K' := \set{(k,0):\ k \in K}$, $R' := (R,S)$ for any $S>0$, and $\Delta':= \Delta \times \Gamma$ for any natural $\Gamma \subseteq [0,\infty)^n$.
	\end{enumerate}
\end{lemma}

\begin{proof}
	Part (1) is a consequence of the discussion in \cref{gen_mult_germs_section}.  For part (2), note that $S^{\tau'} = S^\tau \times H(\log S)$ and $S^{\tau'}_p = S^\tau_p \times H(\log S)$, so that $\|G\|_{\tau'} = \|F\|_\tau$.
\end{proof}

\begin{nrmks} \label{combinatorics}
	We also need the following (crude) estimates from combinatorics.  Let $n \in \NN$.
	\begin{enumerate}
		\item Given $k \in \NN$, the number of elements $\beta \in \NN^n$ such that $\sum\beta = k$ is bounded above by $k^n$, because $\{\beta \in \NN^n:\ \sum\beta = k\} \subseteq \{1, \dots, k\}^n$.
		\item For $\beta \in \NN^n$ and $k \in \NN$, the number of ways to write $\beta$ as the sum of at most $\sum\beta$ many nonzero elements in $\NN^n$ is bounded above by $2^{n\sum\beta}$.  To see this, note that if $n=1$, then each such sum corresponds to a strictly increasing $k$-tuple $0 < a_1 < \cdots < a_k = \beta$ with $k \le \beta = \sum\beta$; so there are at most $2^\beta = 1+ \sum_{k=1}^\beta \binom{\beta}{k}$ many ways to write $\beta$ as the sum of at most $\beta$ many nonzero natural numbers.  The claim for general $n$ follows.
		\item For $\gamma \in \NN^{n'}$ and $k \in \NN$, denote by $N(\gamma,k)$ the number of ways to write $\gamma$ as the sum of exactly $k$ many nonzero elements in $\NN^{n'}$.  Then for any nonzero $\beta \in \NN^n$, we have $$N\left(\gamma,\sum\beta\right) = \sum_{\gamma^1 + \cdots + \gamma^n = \gamma}  \left(\prod_{j=1}^n N\left(\gamma^j,\beta_j\right)\right).$$
	\end{enumerate}
\end{nrmks}

Let $\sigma$ be an infinitesimal substitution as in \cref{subst_expls}(5); in particular, we may assume that $n > 0$.  Let $\tau = (K,R,r,\theta,\Delta) \in \T_m$ and $\rho \in (0,\infty)^n$, and let $F = \sum_{\beta \in \NN^n} F_\beta Y^\beta \in \Pc{C}{X^*,Y}_{\tau,\rho}$.  By \cref{mixed_directed_lemma}(2), there exist $\tau' \in \T_{m'}$ and $\rho' \in (0,\infty)^{n'}$ such that 
\begin{equation}\label{prime_est}
	\|\sigma(Y_j)\|_{\tau',2^{n'+1}\rho'} \le \frac{\rho_j}2 \quad\text{for } j = 1, \dots, n;
\end{equation}
in particular, we have $\tilde{\sigma}\left(S^\tau \times S^{\tau'} \times D(\rho')\right) \subseteq S^\tau \times D(\rho)$.
For each $j$, we write $\sigma(Y_j) = \sum_{\gamma \in \NN^{n'}} G_{j,\gamma}(X') (Y')^\gamma$ with $G_{j,\gamma} \in \Pc{C}{(X')^*}_{\tau'}$, and we write $\tau' = (K',R',r,\theta,\Delta')$ (we can always reduce to the case where $r$ and $\theta$ are the same for both $\tau$ and $\tau'$).
	
By \cref{adding_vars}(2), we have $F_\beta \in \Pc{C}{(X,X')^*}_{\eta_1}$ for each $\beta$, where $$\eta_1 = (K_1,(R,R'),r,\theta,\Delta \times \Delta')$$ with $K_1 := \set{(k,0):\ k \in K}$.  Again by \cref{adding_vars}(2), and by \cref{permutation_lemma}, we have $G_{j,\gamma} \in \Pc{C}{(X,X')^*}_{\eta_2}$ for each $j$ and each $\gamma$, where $$\eta_2 = (K_2,(R,R'),r,\theta,\Delta \times \Delta')$$ with $K_1 := \set{(0,k):\ k \in K'}$.  So from \cref{adding_vars}(1), we get that each $F_\beta$ and each $G_{j,\gamma}$ belongs to $\Pc{C}{(X,X')^*}_\eta$, where $\eta = (L,(R,R'),r,\theta, \Delta \times \Delta')$ with $L:= K_1 \cup K_2$.  Moreover, \cref{adding_vars} implies that $\|F_\beta\|_\eta \le \|F_\beta\|_\tau$ and $\|G_{j,\gamma}\|_\eta \le \|G_{j,\gamma}\|_{\tau'}$ for each $\beta$, $j$ and $\gamma$. 

\begin{prop} \label{infinitesimal_prop}
	For $F \in \Pc{C}{X^*,Y}_{\tau,\rho}$, we have $\sigma F \in \Pc{C}{(X,X')^*,Y'}_{\eta,\rho'}$ such that $\bar{\sigma F} = \bar{F} \circ \tilde\sigma$ and $\|\sigma F\|_{\eta,\rho'} \le A\|F\|_{\tau,\rho}$, for some absolute constant $A>0$.
\end{prop}
	
\begin{proof} 
	A first computation (left to the reader) shows that, for $\beta \in \NN^{n}$, we have $$(\sigma(Y))^\beta = \sum_{\gamma \in \NN^{n'}} H_{\beta,\gamma}(X') (Y')^\gamma,$$  where $$H_{\beta,\gamma}(X') = \sum_{\gamma^1 + \cdots + \gamma^n = \gamma} \left(\prod_{j=1}^n \left(\sum_{\delta^1 + \cdots + \delta^{\beta_j} = \gamma^j} \left(\prod_{p=1}^{\beta_j} G_{j,\delta^p}(X')\right)\right)\right)$$  and each $\gamma^j$ and $\delta^p$ belongs to $\NN^{n'}$ and is nonzero.  Therefore, $$\sigma F(X,X',Y') = \sum_\beta F_\beta(X) \sigma(Y)^\beta = \sum_\gamma \left(\sum_\beta F_\beta(X) H_{\beta,\gamma}(X')\right) (Y')^\gamma.$$  By the above, we have
	\begin{equation} \label{f_beta_norm}
		\|F_\beta\|_\eta \le \frac{\|F\|_{\tau,\rho}}{\rho^\beta}
	\end{equation}
	for each $\beta$, and, setting $\tilde\rho:= 2^{n'+1}\rho'$, 
	\begin{equation}\label{g_gamma_norm}
		\|G_{j,\gamma}\|_{\eta} \le \frac{\|\sigma(Y_j)\|_{\tau',\tilde\rho}}{(\tilde\rho)^\gamma} \le \frac{\rho_j/2}{(\tilde\rho)^\gamma}
	\end{equation}
	for each $j$ and $\gamma$.  Therefore, denoting by $N(\gamma,k)$ the number of ways to write $\gamma$ as the sum of $k$ many nonzero elements in $\NN^{n'}$, we get for each $\beta$ and $\gamma$ from \cref{combinatorics}(3) that
	\begin{align*}
		\|H_{\beta,\gamma}\|_{\eta} & \le \sum_{\gamma^1 + \cdots + \gamma^n = \gamma} \left(\prod_{j=1}^n \left(\sum_{\delta^1 + \cdots + \delta^{\beta_j} = \gamma^j} \left(\prod_{p=1}^{\beta_j} \frac{\rho_j/2}{(\tilde\rho)^{\delta^p}}\right)\right)\right) \\
		&= \sum_{\gamma^1 + \cdots + \gamma^n = \gamma} \left(\prod_{j=1}^n \left(\sum_{\delta^1 + \cdots + \delta^{\beta_j} = \gamma^j} \frac{(\rho_j/2)^{\beta_j}}{(\tilde\rho)^{\gamma^j}}\right)\right) \\
		&= \sum_{\gamma^1 + \cdots + \gamma^n = \gamma} \left(\prod_{j=1}^n \left( \frac{(\rho_j/2)^{\beta_j}}{(\tilde\rho)^{\gamma^j}} N\left(\gamma^j,\beta_j\right) \right)\right) \\
		&= \sum_{\gamma^1 + \cdots + \gamma^n = \gamma} \left(\frac{(\rho/2)^{\beta}}{(\tilde\rho)^{\gamma}} \left(\prod_{j=1}^n N\left(\gamma^j,\beta_j\right)\right)\right) \\
		&= \frac{(\rho/2)^{\beta}}{(\tilde\rho)^{\gamma}} N\left(\gamma,\sum\beta\right). \\
	\end{align*}
	Also, since $N\left(\gamma, k\right) = 0$ for $k > \sum\gamma$, it follows from \cref{combinatorics}(2) that 
	\begin{equation}\label{H_beta_gamma_norm}
		\|H_{\beta,\gamma}\|_{\eta} \le 2^{n' \cdot \sum\gamma} \cdot  \frac{(\rho/2)^{\beta}}{(\tilde\rho)^{\gamma}}.
	\end{equation}
	So for each $\gamma$, we get
	\begin{equation*}
		\left\|\sum_\beta (F_\beta H_{\beta,\gamma}) \right\|_{\eta} 
		\le \|F\|_{\tau,\rho} \frac{2^{n'\cdot\sum\gamma}}{(\tilde\rho)^\gamma} \cdot \sum_\beta \left(\frac12\right)^{\sum\beta}.
	\end{equation*}
	However, since $\sum_\beta \left(\frac12\right)^{\sum\beta} = \sum_{k=0}^\infty \left(\sum_{\sum\beta = k} 1\right) \left(\frac12\right)^k \le C:= \sum_k \frac{k^n}{2^k}$ by \cref{combinatorics}(1), we conclude that 
	\begin{equation*}
		\left\|\sum_\beta (F_\beta H_{\beta,\gamma}) \right\|_{\eta} 
		\le C \cdot \|F\|_{\tau,\rho} \frac{2^{n'\cdot\sum\gamma}}{(\tilde\rho)^\gamma},
	\end{equation*}
	where $C>0$ is an absolute constant.  Finally, since $(\tilde\rho)^\gamma = 2^{(n'+1)\sum\gamma}\cdot (\rho')^\gamma$, we obtain 
	\begin{equation}\label{H_gamma_norm}
		\left\|\sum_\beta (F_\beta H_{\beta,\gamma}) \right\|_{\eta} 
		\le \frac{C \cdot \|F\|_{\tau,\rho}}{2^{\sum\gamma}\cdot (\rho')^\gamma}\ .
	\end{equation}
	Multiplying by $(\rho')^\gamma$ and summing over $\gamma$ therefore yields
	$$\|\sigma F\|_{\eta,\rho'} \le C \cdot \|F\|_{\tau,\rho} \left(\sum_\gamma \left(\frac12\right)^{\sum\gamma}\right) \le C D \|F\|_{\tau,\rho}$$
	for some absolute constant $D>0$, as required.
\end{proof}

\section{Closure properties and o-minimality}  \label{o-min_chapter}

The goal of this section is to verify, for those mixed series of \cref{mixed_section} \textit{with only real coefficients}, the closure properties listed in Paragraphs 1.8 and 1.15 of \cite{MR3349791}.  We will adopt the notations of the latter, and we need to define the real algebras $\A_{m,n,r}$.  So let $m,n \in \NN$, and let $r = (s,t) = (s_1, \dots, s_m, t_1, \dots, t_n) \in (0,\infty)^{m+n}$ be a polyradius of type $(m,n)$.  (While we used the letter $R$ for polyradii in the previous sections to mirror notations  in \cite{Dries:2000mx}, we use the letters $r$, $s$ and $t$ to mirror corresponding notations in \cite{MR3349791}.)
We set $$\T_m^s := \set{\tau = (K,R,\rho,\theta, \Delta) \in \T_m:\ R > s}$$ and define the $\RR$-algebras $$\Pc{R}{X^*,Y}_r:= \bigcup_{\tau \in \T_m^s \atop u > t} \Pc{C}{X^*,Y}_{\tau,u} \cap \Ps{R}{X^*,Y}$$ and $$\bar\A_{m,n,r} := \bigcup_{\tau \in \T_m^s} \set{\bar F \rest{(-\infty,s) \times D(t)}:\ F \in \Pc{C}{X^*,Y}_{\tau,t} \cap \Ps{R}{X^*,Y}}.$$  Recall from \cite[Notation 1.7]{MR3349791} the following definitions:
\begin{align*}
	I_{m,n,r} & :=(0,s_{1})\times\cdots\times(0,s_{m})\times(-t_{1},t_{1})\times\cdots\times(-t_{n},t_{n})\\
	\hat{I}_{m,n,r} & :=[0,s_{1})\times\cdots\times[0,s_{m})\times\left(-t_{1},t_{1}\right)\times\cdots\times\left(-t_{n},t_{n}\right).
\end{align*}
In particular,  each $\bar f \in \bar\A_{m,n,r}$ defines a continuous function $f:\hat I_{m,n,r} \into \RR$ by setting $$f(x,y):= \bar f(\log x,y);$$ this $f$ is real analytic on $I_{m,n,r}$.  We let $\A_{m,n,r}$ be the $\RR$-algebra of all such functions obtained from $\Pc{R}{X^*,Y}_r$, and we define the $\RR$-algebra homomorphism $T_{m,n,r}:\A_{m,n,r} \into \Pc{R}{X^*,Y}_r$ by letting $T_{m,n,r}f$ be the (by quasianalyticity) unique $F \in \Pc{R}{X^*,Y}_r$ such that $f(x,y) = \bar F(\log x,y)$.
We leave it to the reader to verify Properties (1)--(5) and (8) of \cite[Paragraph 1.8]{MR3349791}.  Properties (6) and (7) follow from \cref{mixed_directed_lemma}(4,5).

Let $\{\A_{m,n}:\ m,n \in \NN\}$ be the corresponding family of algebras of germs, as defined in \cite[Section 1.2]{MR3349791}.  By \cref{translation_prop}, every $f \in \A_{m,n,r}$ is $\A$-analytic, as defined in \cite[Definition 1.10]{MR3349791}.

It now remains to verify the properties listed in \cite[Paragraph 1.15]{MR3349791} for the corresponding family $\A$ of algebras of germs.  Property (1) there is obvious here; Property (3) follows from \cref{permutation_lemma}; Property (5) follows from \cref{infinitesimal_prop}; and Property (6) follows from \cref{weierstrass}.  The remaining properties are handled below; we fix arbitrary $\tau = (K,R,r,\theta,\Delta) \in \T_m$ and $\rho \in (0,\infty)^n$.

\subsection*{Monomial divison}  
Let $F \in \Pc{C}{X^*,Y}_{\tau,\rho}$.

First, we  let $\alpha > 0$ and assume that $F = X_1^\alpha G$ with $G \in \Ps{C}{X^*,Y}$.  We write $F = \sum_{\beta \in \NN^n} F_\beta(X) Y^n$ and $G = \sum_{\beta \in \NN} G_\beta(X)Y^n$ with each $F_\beta \in \Pc{C}{X^*}_\tau$ and each $G_\beta \in \Ps{C}{X^*}$, and we set $\tau':= (K,R,s,\theta,\Delta)$ for some fixed (but arbitrary) $s \in (1,r)$.  Then, by \cref{permutation_lemma} and \cref{mon_div_lemma}, there exist $C>0$ (depending only on $\frac sr$)  such that $\|G_\beta\|_\tau' \le C\|F_\beta\|_\tau$, for each $\beta$.  It follows that $\sum_\beta \|G_\beta\|_{\tau'} \rho^n \le C \|F\|_{\tau,\rho}$, so that $G \in \Pc{C}{X^*,Y}_{\tau',\rho}$.

Second, we let $n \in \NN$ and assume that $F = Y_1^n G$ with $G \in \Ps{C}{X^*,Y}$.  Then by definition, we have $\|F\|_{\tau,\rho} = \rho_1^n \|G\|_{\tau,\rho}$, so that $G \in \Pc{C}{X^*,Y}_{\tau,\rho}$.

Putting together the two cases discussed here proves Property (2) of \cite[Paragraph 1.15]{MR3349791}.
	
\subsection*{Setting a variable equal to 0}  
Let $F \in \Pc{C}{X^*,Y}_{\tau,\rho}$, and write $F = \sum_{\beta \in \NN^n} F_\beta(X) Y^n$ with each $F_\beta \in \Pc{C}{X^*}_\tau$.  Applying \cref{fixing_lemma} with $\nu = 1$ and $a = -\infty = \log 0$ followed by \cref{permutation_lemma}, we get that $F_{\beta,0}:= F_\beta(X_1, \dots, X_{n-1},0) \in \Pc{C}{(X_1, \dots, X_{n-1})^*,Y}_{\mu(a)}$ with $\|F_{\beta,0}\|_{\mu(a)} \le \|F_\beta\|_\tau$, for each $\beta \in \NN^n$.  Hence $F_0:= F(X_1, \dots, X_{n-1},0,Y)$ belongs to $\Pc{C}{(X_1, \dots, X_{n-1})^*,Y}_{\mu(a),\rho}$; this proves Property (4) of \cite[Paragraph 1.15]{MR3349791}.

\subsection*{Blow-up charts}
Here we refer to the blow-up charts (1)--(5) of \cite[Definition 1.13]{MR3349791}.  Closure under blow-up charts (1) (regular blow-ups) is proved by \cref{blow-up_prop}; closure under blow-up charts (2) (singular blow-ups) is similar, but easier and left to the reader.  Blow-up charts (3) are infinitesimal substitutions and thus are handled by \cref{infinitesimal_prop}.  For blow-up charts (4), note that $\Pc{C}{X^*,Y}_\G \subseteq \Pc{C}{(X_1, \dots, X_{m+1})^*,Y}_\G$ by \cref{mixed_directed_lemma}(6); so closure under these blow-up charts follows from closure under blow-up charts (2) and \cref{infinitesimal_prop}.  Closure under blow-up charts (5) follows again from \cref{infinitesimal_prop}.  This proves Property (7) of \cite[Paragraph 1.15]{MR3349791}.\medskip

As the results of \cite{MR3349791} do not make use of the Weierstrass Preparation Theorem (which is in general not available in the quasianalytic setting), we may dispense with proving this property here.

\begin{proof}[Proof of Main Theorem]
	The previous discussion implies that our system $\A$ of algebras satisfies Conditions (1) and (4) of \cite[Proviso 1.20]{MR3349791}.  Moreover, Condition (2) is implied by \cref{translation_prop}, while Condition (3) follows from \cref{multi_algebra_prop}.  So the theorem follows from \cite[Theorems A and B]{MR3349791}.  Finally, note that it suffices to add the reciprocal function to obtain quantifier elimination, as all real powers with nonnegative exponents are already in the language $\la_{\G^*}$.
\end{proof}



\end{document}